\documentclass {siamltex}
\usepackage[english]{babel}
\usepackage{amsfonts}

\usepackage{amsmath}
\usepackage{amssymb,amsbsy}
\usepackage{amscd}
\usepackage{latexsym}
\usepackage{float}
\usepackage{graphicx}
\usepackage{setspace}
\usepackage{fancyhdr}
\usepackage{color}
\usepackage{stmaryrd,algorithm}

\newcommand{\jump}[1]{\llbracket #1 \rrbracket }

\newcommand{\OmegaC}{{C}}

\newcommand{\bbR}{\mathbb{R}}

\usepackage{subfigure}
\usepackage[bordercolor=white, color=white, colorinlistoftodos]{todonotes}

\newtheorem{assumption}{Assumption}[section]

\usepackage{color}
\definecolor{myorange}{rgb}{0.9568,0.4941,0.1961}
\definecolor{myred}{rgb}{0.9098,0.1294,0.2078}
\definecolor{myblue}{rgb}{0.0352,0.4981,0.6509}
\definecolor{myhyperblue}{rgb}{0.1607,0.3922,0.9}
\definecolor{mygreen}{rgb}{0.2235,0.6353,0.2588}
\definecolor{mygrey}{rgb}{0.3,0.3,0.3}

\graphicspath{{Figures/}}
\title{%\textcolor{black}
{Augmented Lagrangian finite element methods for
  contact problems}}% using Lagrange multipliers}

\author{Erik Burman\thanks{Department of Mathematics, 
University College London, Gower Street, London, 
UK--WC1E  6BT, 
United Kingdom; ({\tt e.burman@ucl.ac.uk})}
\and Peter Hansbo\thanks{Department of Mechanical Engineering, J\"onk\"oping University,
SE-55111 J\"onk\"oping, Sweden; ({\tt peter.hansbo@ju.se})}
\and Mats G. Larson \thanks{Department of Mathematics and Mathematical Statistics, Ume{\aa} University, 
SE-901 87 Ume{\aa}, Sweden; ({\tt mats.larson@math.umu.se})}
}
 
\begin{document}

\maketitle

\begin{abstract}
We propose two different Lagrange multiplier methods for contact
problems derived from the augmented Lagrangian variational formulation. Both
the obstacle problem, where a constraint on the solution is imposed in
the bulk domain and the Signorini problem, where a lateral contact
condition is imposed are considered. We consider
both continuous and discontinuous approximation spaces for the
Lagrange multiplier. In the latter case the method is unstable and a
penalty on the jump of the multiplier must be applied for
stability. We prove the existence and uniqueness of discrete
solutions, best approximation estimates and convergence estimates that
are optimal compared to the regularity of the solution.
\end{abstract}

% ********  NEW THEOREMS  ********
%\newtheorem{theorem}{Theorem}[section]
%\newtheorem{lemma}{Lemma}[section]
\newtheorem{cor}{Corollary}[section]
\newtheorem{remark}{Remark}[section]
% **********************************

\section{Introduction}
We consider the Signorini problem, find $u$ and $\lambda$ such that
\begin{equation}\label{signorini}
\begin{array}{rcl}
-\Delta u &=& f \mbox{ in } \Omega \\
u &= & 0  \mbox{ on } \Gamma_D \\
u \leq 0,\; \lambda \leq 0, \; u \,\lambda &=& 0 \mbox{ on } \Gamma_C,
\end{array}
\end{equation}
or the obstacle problem
\begin{equation}\label{obstacle}
\begin{array}{rcl}
-\Delta u -\lambda &=& f \mbox{ in } \Omega \\
u &= & 0  \mbox{ on } \partial \Omega\\
u \leq 0,\; \lambda \leq 0, \; u \,\lambda &=& 0 \mbox{ in } \Omega.
\end{array}
\end{equation}
Here $\Omega\subset \bbR^d$, $d=2,3$ is a bounded polyhedral
(polygonal) domain and $f\in L_2(\Omega)$. It is well known that these problems admit unique solutions $u \in
H^1(\Omega)$. This follows from the theory of Stampacchia applied to
the corresponding variational inequality (see for instance
\cite{HHN96}). 

From a mechanical point of view, these equations model the deflection of a membrane in isotropic tension under the load $f$, assuming small deformations.
The membrane is either in contact with an obstacle on part of the boundary, (\ref{signorini}), or in the interior of the membrane, (\ref{obstacle}), preventing positive displacements $u$. 
In both cases the Lagrange multiplier has the interpretation of a distributed reaction force enforcing the contact condition $u\leq 0$.

\section{Finite element discretization}

Our aim in this paper is to design a consistent penalty method for contact
problems that can easily be included in a standard
Lagrange-multiplier method, without having to resort to the solution
of variational inequalities. We consider two different choices for the
multiplier spaces, either a stable choice or an
unstable choice where a stabilization term is needed to ensure the
stability of the formulation. In
the latter case
we add a penalty on the jump of the multiplier over element
faces in the spirit of \cite{BH10a,BH10b}.

There exists a large body of litterature treating finite
element methods for contact
problems \cite{BHR78,KO88,GleT89,BB00,BBB03,BBR03,Wriggers2007,WPGW12,CH13a}. Discretization of \eqref{signorini} is usually performed on
the variational inequality or using a penalty method. The first case however leads to some nontrivial
choices in the construction of the discretization spaces in order to
satisfy the nonpenetration condition and associated inf-sup conditions
and until recently it has proved difficult to
obtain optimal error estimates \cite{HR12, DH16}. The latter case, on the
other hand leads to the usual
consistency and conditioning problems of penalty methods. Another approach proposed by Hild and
Renard \cite{HR10} is to use a
stabilized Lagrange-multiplier in the spirit of Barbosa and Hughes \cite{BH92}.
As a further development one may use the reformulation of the contact condition
\begin{equation}\label{cond1}
\lambda = - \gamma^{-1} [u - \gamma \lambda ]_+
\end{equation}
where $[x]_+ = \max(0, x)$, introduced by Alart and Curnier \cite{AC91} in an augmented Lagrangian
framework.
Using the close relationship between the Barbosa--Hughes met\-hod and
Nitsche's method \cite{Nit71} discussed by Stenberg \cite{Sten95}, this method was
then further developed in the elegant Nitsche-type formulation for the
Signorini problem
introduced by Chouly, Hild and Renard  \cite{CH13b,CHR15}. In these works
optimal error estimates for the above model problem were obtained for
the first time. 

Using the notation $\left<u,v \right>_C$ for the $L_2$ inner product over $C$ we have
in the case of the Signorini problem \eqref{signorini} that $C$ corresponds
to $\Gamma_C$, the boundary part where the contact conditions hold and
\[
\left<u,v \right>_C := \int_{\Gamma_C} uv  ~\mbox{d}s ,
\] 
while for the obstacle
problem  \eqref{obstacle} $C \equiv \Omega$ and 
\[
\left<u,v \right>_C := \int_{\Omega} uv  ~\mbox{d}x.
\] 
Finally, we define $\| v\|_C := \left<v,v \right>_C^{1/2}$.
With this notation, the augmented Lagrangian multiplier seeks stationary points to the functional
\begin{equation}\label{functional}
\mathfrak{F} (u,\lambda) := \frac12 a(u,u) + \frac{1}{2\gamma}\| [u - \gamma \lambda ]_+\|^2_C -\frac{\gamma}{2}\|\lambda\|_C^2 ,
\end{equation}
cf. Alart and Curnier \cite{AC91}.
Observe that formally the stationary points are given by $(u,\lambda)$ such
that
\begin{equation}\label{eq:stat_point1}
\begin{array}{rcl}
a(u,v) + \left<\gamma^{-1}  [u - \gamma \lambda ]_+, v \right>_C &=&
                                                         (f,v)_\Omega\\
\left< \lambda + \gamma^{-1} [u - \gamma \lambda ]_+, \mu \right>_C&=& 0
\end{array}
\end{equation}
for all $(v,\mu)$, or by substituting the second equation in the first
\begin{equation}\label{eq:stat_point2}
\begin{array}{rcl}
a(u,v) - \left< \lambda, v \right>_C &=&
                                                         (f,v)_\Omega\\
\left< \gamma \lambda + [u - \gamma \lambda ]_+, \mu \right>_C&=& 0.
\end{array}
\end{equation}
Observing now that the contact condition equally well can be written
on the primal variable as $u = -[\gamma \lambda - u]_+$ we get by
adding and subtracting $u$ in the second equation of
\eqref{eq:stat_point2}
\begin{equation}\label{eq:stat_point3}
\begin{array}{rcl}
a(u,v) - \left< \lambda, v \right>_C &=&
                                                         (f,v)_\Omega\\
\left<u + [\gamma \lambda - u ]_+, \mu \right>_C&=& 0.
\end{array}
\end{equation}
In this paper we consider two different method, resulting from this
approach. The first formulation is the straightforward discretization
of \eqref{eq:stat_point1} resulting in a method that gives the
stationary points of the functional \eqref{functional} over the
discrete spaces. The second formulation is a discretization of
\eqref{eq:stat_point3} that is chosen for its closeness to the
standard Lagrange multiplier method for the imposition of Dirichlet
boundary conditions.

We consider discretization either with a choice
of approximation spaces that results in a stable approximation, or a
choice that is stable only with an added stabilizing term. Here we
consider stabilization based on the interior
penalty stabilized Lagrange multiplier method introduced by Burman and
Hansbo \cite{BH10a} for solving elliptic interface problems. The
appeal of this latter approach is that we may use the lowest order 
approximation spaces where the displacement is piecewise linear and
the multiplier constant per element (or element side). When
considering the Signorini problem \eqref{signorini} 
these spaces match the regularity of the physical problem perfectly and therefore in some sense is the
most economical choice.  Contact problems also present non trivial
quadrature problems so that in practice it can be very difficult to
integrate the terms of the formulation to a sufficient accuracy to get
optimal accuracy when
higher order interpolations are used. Herein we will assume that
integration can be performed exactly on the interface between the
contact and non-contact subdomain.

%An advantage of the
%Lagrange multiplier method compared to Nitsche's method is that the
%former case can be applied also for general obstacle problem.
For an alternative stabilization method of Barbosa--Hughes type in the augmented Lagrangian setting, see Hansbo, Rashid, and Salomonsson \cite{HRS16}.

We assume that $\{ \mathcal{T}\}_h$ is a family of quaisuniform meshes of $\Omega$,
such that the mesh is fitted to the zone $C$. That is $C$ is a subset
of boundary element faces of simplices $K$ such that $K \cap \Gamma_C \ne \emptyset$, $F:= \partial K \cap \Gamma_C$ 
$\mathcal{T}_C:=\{F\} $, $C :=
\cup_{F \in \mathcal{T}_C}$ with $C \subset \mathbb{R}^{d-1}$ for the
Signorini problem. For the obstacle
problem $C$ is defined by $\Omega$ and hence $\cup_{K \in
  \mathcal{T}}=: C \subset \mathbb{R}^{d}$ and $\mathcal{T}_C \equiv
\mathcal{T}$. Below we will denote the elements of $\mathcal{T}_C$ by
$K$ in both cases.
We
define $V_h$ to be the space of $H^1$-conforming
functions on $\mathcal{T}$, satisfying the homogeneous boundary
condition of $\Gamma_D$.
\[
V^k_h := \{v_h \in H^1(\Omega): v\vert_{\Gamma_D} = 0; v\vert_K \in
\mathbb{P}_k(K),\, \forall K \in \mathcal{T} \},
\]
where $\mathbb{P}_k(K)$ denotes the set of polynnomials of order less
than or equal to $k$ on the simplex $K$.
Whenever the superscript is dropped we refer to the generic space of
order $k$.
For the multipliers we introduce the space $\Lambda_h$ defined 
as the space piecewise polynomials of order less than or equal to $l$ defined on 
$\OmegaC$. % := \cup_{K \in \mathcal{T}_C}$.
\[
\Lambda^{l}_h := \{\mu_h \in L^2(C): \mu_h \vert_{K} \in
\mathbb{P}_{l} (K), \forall K \in \mathcal{T}_C\}.
\]
Whenever $l=k-1$ the superscript is dropped.
We will detail the case of discontinuous multipliers, but all
arguments below are valid also in case the Lagrange multiplier is
approximated in the space of continuous functions, $\Lambda^l_h \cap
C^0(\OmegaC)$, $l \ge 1$, in this case no stabilization is
necessary. The differences in the analysis will be outlined. 

Both
formulations that we consider herein take the form: Find $(u_h,\lambda_h) \in
V_h \times \Lambda_h$ such that
\begin{equation}\label{FEM}
a(u_h,v_h) + b[(u_h,\lambda_h);(v_h,\mu_h)] =(f,v_h)_\Omega \quad \forall (v_h,\mu_h)
\in V_h \times \Lambda_h
\end{equation}
where $(\cdot,\cdot)_\Omega$ denotes the standard $L^2$-inner product,
$a(u_h,v_h) := (\nabla u_h,\nabla v_h)_{\Omega}$ and the methods are distinguished by
the definition of the form $b[\cdot;\cdot]$ that acts only in the zone
where contact may occur. The stabilization will be included in the
form $b[\cdot;\cdot]$. As already pointed out this term is
necessary if the choice $V_h \times \Lambda_h$, does not satisfy the inf-sup
condition. In our framework, this is the case where the multiplier is
discontinuous over element faces. 
In this paper we will focus on a stabilization
using a penalty on the jumps over element faces of the multiplier
variable in the spirit of \cite{BH10a,BH10b},
\begin{equation}\label{def:s}
s(\lambda_h,\mu_h):=\sum_{F \in \mathcal{F}_{C}} \delta \gamma \int_F h
\jump{\lambda_h}\jump{\mu_h} ~\mbox{d}s,
\end{equation}
where $\delta>0$ is a parameter, $\jump{x}\vert_F$ denotes
the jump of the quantity $x$ over the face $F$ and $\mathcal{F}_C$ denotes the set of interior element
faces of the elements in $\mathcal{T}_C$. The semi-norm associated with the
stabilization operator will be defined as $|\cdot|_s := s(\cdot,\cdot)^{\frac12}$.

We will also below use the compact notation
\[
A_h[(u_h,\lambda_h),(v_h,\mu_h)]:=a(u_h,v_h) + b[(u_h,\lambda_h);(v_h,\mu_h)]
\]
and the associated formulation, find $(u_h,\lambda_h) \in V_h \times
\Lambda_h$ such that
\begin{equation}\label{CompFEM}
A_h[(u_h,\lambda_h),(v_h,\mu_h)] = (f,v_h)_\Omega,\mbox{ for all }
(v_h,\mu_h) \in V_h \times
\Lambda_h.
\end{equation}
We will now specify two different choices of $b[\cdot;\cdot]$ leading
to two different Lagrange-multiplier methods. 
\newline

{\flushleft\bf FORMULATION 1}:
In the first formulation we use the original formula for the contact
condition proposed by Alart and Curnier, $\lambda = -\gamma^{-1} [ u-\gamma \lambda]_+$
\begin{align}\nonumber
b[(u_h,\lambda_h);(v_h,\mu_h)]:= {}& \left<
 \gamma^{-1}[u_h - \gamma \lambda_h]_+, v_h\right>_{C} \\ \nonumber
{}& +  
\left<
 \gamma^{-1}[u_h - \gamma \lambda_h]_+, \gamma  \mu_h\right>_{C} \\
{}&+\left<\gamma
  \lambda_h, \mu_h \right>_{C} 
+ s(\lambda_h,\mu_h)\label{stab_form0}
\end{align}
or, writing the nonlinearity as the derivative of a quadratic form,
and using the notation $P_{\gamma\pm}(u_h,\lambda_h):=\pm(u_h - \gamma \lambda_h)$
\begin{align}\nonumber
b[(u_h,\lambda_h);(v_h,\mu_h)]:={}&\left<
 \gamma^{-1}[P_{\gamma+}(u_h,\lambda_h)]_+,P_{\gamma+}(v_h,\mu_h)\right>_{C}\\
 {}& - \left<\gamma \mu_h, \lambda_h \right>_{C}
-  s(\lambda_h,\mu_h),\label{stab_form1}
\end{align}
with $\gamma>0$ a parameter to determine.
In this case the finite element formulation corresponds
to the approximate solutions of \eqref{eq:stat_point1} in the finite element space.
\newline

{\flushleft\bf FORMULATION 2}:
In the second formulation we use a reformulation of the contact
condition on the displacement variable, $u = -[\gamma \lambda -
u]_+$ to obtain the semi-linear form
\begin{align}\nonumber
b[(u_h,\lambda_h);(v_h,\mu_h)]:= {}& - \left<
  \lambda_h,v_h\right>_{C} +\left<
  \mu_h,u_h  \right>_{C} \\
{}& +  \left<
  \mu_h, [P_{\gamma-}(u_h,\lambda_h)]_+\right>_{C} + s(\lambda_h,\mu_h), \label{infsup_stabform1}
\end{align}
with $\gamma>0$ a parameter to determine.
In this case the finite element formulation corresponds
to the approximate solutions of \eqref{eq:stat_point3} in the finite element space.
\subsection{Alternative formulations}
In both formulation 1 and 2 above it is possible to derive an
alternative formulation of the same method using the relation
\[
[P_{\gamma-}(u_h,\lambda_h)]_+ = [P_{\gamma+}(u_h,\lambda_h)]_+-P_{\gamma+}(u_h,\lambda_h).
\]
Considering the form \eqref{stab_form1} and adding and
subtracting $P_{\gamma+}(u_h,\lambda_h)$ in the nonlinear term we have
the alternative form (omitting the stabilization term)
\begin{align}\nonumber
b[(u_h,\lambda_h);(v_h,\mu_h)]= {}&\left<
 \gamma^{-1}[P_{\gamma+}(u_h,\lambda_h)]_+,P_{\gamma+}(v_h,\mu_h)\right>_{C}
 - \left<\gamma \mu_h, \lambda_h \right>_{C} \\ \nonumber
= {}&\left<
 \gamma^{-1}([P_{\gamma+}(u_h,\lambda_h)]_+-P_{\gamma+}(u_h,\lambda_h)),P_{\gamma+}(v_h,\mu_h)\right>_{C}\\ \nonumber
{}&+\left<
 \gamma^{-1}P_{\gamma+}(u_h,\lambda_h),P_{\gamma+}(v_h,\mu_h)\right>_{C}
 - \left<\gamma \mu_h, \lambda_h \right>_{C} \\ \nonumber
 = {}& - \left<
  \lambda_h,v_h\right>_{C} +\left<
  \mu_h,u_h  \right>_{C} +\gamma^{-1} \left< u_h,v_h\right>_{C}\\
{}& +\left<
 \gamma^{-1}([P_{\gamma-}(u_h,\lambda_h)]_+,P_{\gamma+}(v_h,\mu_h)\right>_{C}.\label{eq:alt_2}
\end{align}
Similarly for formulation 2 we obtain in \eqref{infsup_stabform1}
omitting for simplicity the stabilization term
\begin{align}\nonumber
b[(u_h,\lambda_h);(v_h,\mu_h)]= {}& - \left<
  \lambda_h,v_h\right>_{C} +\left<
  \mu_h,u_h  \right>_{C} \\ \nonumber
{}& +  \left<
  \mu_h,
  [P_{\gamma-}(u_h,\lambda_h)]_++P_{\gamma+}(u_h,\lambda_h)-P_{\gamma+}(u_h,\lambda_h)\right>_{C}
\\ \nonumber
= {}& - \left<
  \lambda_h,v_h\right>_{C} +\gamma \left<
  \mu_h,\lambda_h \right>_{C} 
  \\ {}& +
 \left<
  \mu_h,
  [P_{\gamma+}(u_h,\lambda_h)]_+\right>_{C}.\label{eq:alt_1}
\end{align}
We see that this semi-linear form corresponds to a discretization of
\eqref{eq:stat_point2}.

The methods defined by \eqref{infsup_stabform1}
and \eqref{eq:alt_1} or \eqref{stab_form1} and \eqref{eq:alt_2}
respectively are equivalent, but if during the solution process the linear
and nonlinear parts are separated in the nonlinear solver, one can
expect the different formulations to have different behavior and give
rise to different sequences of approximations in the iterative
procedure.

\section{Technical results}
%\HOX{Introduce notation in the beginning, define the spaces. Check for consistency}
Here we will collects some useful elementary results. First recall the
following inverse inequalities and trace inequalities (for a proof see, e.g.,
\cite{Arn82})
\begin{equation}\label{inverse}
\|\nabla u_h\|_K \leq C_ i h^{-1} \|u_h\|_K, \quad \forall u_h \in
V_h 
\end{equation}
\begin{equation}\label{trace}
\|u\|_{\partial K} \leq C_T (h^{-\frac12} \|u\|_K + h^{\frac12}
\|\nabla u\|_K), \quad \forall u \in H^1(K)
\end{equation}
\begin{equation}\label{trace_disc}
\|u_h\|_{\partial K} \leq
C_T h^{-\frac12} \|u_h\|_K,\quad \forall u_h \in
V_h
\end{equation}
Similar inequalities hold for functions in $\Lambda_h$ and we will use
them without making any distinction between the two cases.
We let $\pi_0:L^2(C) \rightarrow
\Lambda^0_h$ denote the standard $L^2$ projection onto $\Lambda^0_h$ and we observe
that there holds, by standard approximation properties of the
projection onto constants (and a trace inequality in the case of
lateral contact),
\[
\|(1-\pi_0) v_h\|_C \leq c_0 h^{s} \|\nabla v_h\|_{\Omega}
\]
with $s=1$ for the Obstacle problem where $C \subset \Omega$ and
$s=\tfrac12$ for the Signorini problem where $C \subset \partial
\Omega$. Similarly we define $\pi_l:L^2(C) \rightarrow
\Lambda^l_h \cap C^0(\bar C)$ and note that the corresponding
inequality holds for $\pi_1$
\[
\|(1-\pi_1) v_h\|_C \leq c_1 h^{s} \|\nabla v_h\|_{\Omega}.
\]
We also observe for future reference that $\| u\|_C \leq
C \|u\|_{H^1(\Omega)}$ in both cases. 

For the analysis below it is useful to introduce an indicator function
for the contact domain $C$ defined on the space $V_h$.
Let $\xi_h$ denote a finite element function such
that $\xi_h \in V^1_h$ with $\xi_h(x) = 0$ for nodes in $(\bar \Omega
\setminus \bar C) \cup \bar \Gamma_D$, that is nodes outside the contact zone. For all other nodes
$x_i \in K$ with $K \subset \mathcal{T}_C$, $x_i \not \in \bar \Gamma_D$, $\xi_h(x_i)=1$. The
following bound is well known, see for instance \cite{BS16}
\begin{equation}\label{eq:xicont}
\exists c_\xi\in \mathbb{R}^+ \mbox{ such that } c_\xi \|\mu_h\|_C
\leq \|\xi_h^{\frac12} \mu_h\|_C,\quad  \forall \mu_h \in
\Lambda^l_h, \quad l \ge 0.
\end{equation}

Stability of the method will rely on the satisfaction of the following
assumption:
\begin{assumption}\label{ass:infsup}
There exists $c_D \in [0,1)$ such that for all $\mu_h \in \Lambda_h$ there holds
\[
\|(1-\xi_h)  \mu_h \|_{C} \leq c_D \|\mu_h\|_{C}.
\]
\end{assumption}
The assumption % \ref{ass:infsup} 
holds whenever there exists a
quadrature rule on the simplex with positive weights and only interior
quadrature points. This is easily shown by observing that
\begin{align*}
\|(1-\xi_h)  \mu_h \|_{C}^2 &= \sum_{K \in \mathcal{T}_C} \sum_{i \in
  \mathcal{Q}_K}
(1-\xi_h(x_i))^2  \mu_h(x_i)^2 \omega_i 
\\
&\leq
\max_{K\in \mathcal{T}_C}( \max_{i \in \mathcal{Q}_K} (1-\xi_h(x_i))^2 \sum_{K \in \mathcal{T}}
\sum_{i \in \mathcal{Q}_K}  (\mu_h(x_i))^2\omega_i 
\\
&= c_D^2  \|\mu_h\|_{C}^2 
\end{align*}
where $\mathcal{Q}_K$ is a set of integers indexing the quadrature
points in $K$ and $$c_D \equiv \max_{K\in \mathcal{T}_C}( \max_{i \in
  \mathcal{Q}_K} (1-\xi_h(x_i))^2.$$ Since $\xi_h$ is zero only on the
boundary of $C$ and no points $x_i \in \mathcal{Q}_K$ are on the
boundary we conclude that $c_D<1$.

This is a very mild condition, on triangles it has been showed to hold
at least up to integration degree $23$, see \cite{TWB05, ZCL09}. It
follows that for the Signorini problem in three dimensions and the
obstacle problem in two space dimensions the analysis holds at least
up to $k=12$. For the lowest order case where the multipliers are constant
per element it is straightforward to show that  $c_D \leq 1/2$ if $C \subset \mathbb{R}^2$ and $c_D \leq \tfrac13$
if $C \subset \mathbb{R}^3$.

\begin{lemma}\label{lem:monotone}
Let $a,b \in \mathbb{R}$; then there holds
\[
([a]_+-[b]_+)^2 \leq ([a]_+ - [b]_+) (a - b),
\]
\[
|[a]_+-[b]_+| \leq |a-b|.
\]
\end{lemma}
\begin{proof}
Expanding the left hand side of the expression we have
\[
[a]_+^2 +[b]_+^2 - 2 [a]_+[b]_+ \leq  [a]_+ a + [b]_+ b -  a [b]_+
-  [a]_+b =  ([a]_+ - [b]_+) (a - b).
\]
For the proof of the second claim, this is trivially true in case both $a$
and $b$ are positive or negative. If $a$ is negative and $b$ positive then
\[
|[a]_+-[b]_+| = |b| \leq |b-a|
\]
and similarly if $b$ is negative and $a$ positive
\[
|[a]_+-[b]_+| = |a| \leq |b-a|.
\]
\end{proof}
\begin{lemma}(Continuity of $b[\cdot;\cdot]$)\label{bcont}
The forms \eqref{infsup_stabform1} and \eqref{stab_form1} satisfy
\begin{align*}
&|b[(u_1,\lambda_1);(v,\mu)] - b[(u_2,\lambda_2);(v,\mu)]|
\\ 
&\qquad \leq
(\gamma^{-\frac12} \| (u_1-u_2)\|_{H^1(\Omega)} +
\gamma^{\frac12}\|\lambda_1-\lambda_2\|_C) (\gamma^{-\frac12} \|v\|_C
+ \gamma^{\frac12} \|\mu\|_C)
\\
&\qquad \qquad
+|\lambda_1-\lambda_2|_s|\mu|_s.
\end{align*}
\end{lemma}
\begin{proof}
Immediate by the definitions of $b[\cdot;\cdot]$, the second
inequality of Lemma \ref{lem:monotone}, the Cauchy-Schwarz inequality
and the assumptions on $\|\cdot\|_C$.
\end{proof}

Next we define the local averaging interpolation operator $I_{cf}:\Lambda_h
\rightarrow \Lambda_h \cap C^0(\OmegaC)$ such that for every
Lagrangian node $x_i \in \mathcal{T}_C$
\[
 I_{cf} \lambda_h(x_i) = \kappa_i^{-1} \sum_{K:x_i \in K}
 \lambda_h(x_i),
\]
where $\kappa_i$ denotes the cardinality of the set $\{K \subset \mathcal{T}_C:x_i \in K\}$. Observe that since $\xi_h \in V_h^1$, for any $\mu_h \in \Lambda_h$ there are functions
$R_\mu$ in
$V_h$ such that $R_\mu\vert_{C} = I_{cf} \xi_h \mu_h$. We recall
the following interpolation result between discrete spaces:
\begin{proposition}\label{lem:jump_cont}
For all $\mu_h \in
\Lambda_h$ there holds
\[
\|\xi_h \mu_h -  I_{cf}(\xi_h \mu_h)\|_{C} \leq c_s \|h^{\frac12}
\jump{\mu_h}\|_{\mathcal{F}_C}, \quad \|I_{cf}\mu_h\|_{C} \leq
c_{cf} \|\mu_h\|_C
\]
and
\[
|\mu_h|^2_s \leq C \delta \|\mu_h\|^2_C.
\]
\end{proposition}
\begin{proof}
For a proof of the first inequality we refer to \cite[Lemma 5.3]{BE07}. The
second inequality is immediate by applying the trace inequality
\eqref{trace_disc} to each term in the definition \eqref{def:s} 
of $s(\cdot,\cdot)$.
\end{proof}
\begin{lemma}\label{lem:lift}
Let $r_h \in \Lambda_{h} \cap C^0(\OmegaC)$, then there exists $R_h \in V_h$ such that
$R_h\vert_{C} = \xi_h r_h$ and $\|R_h\|_{H^1(\Omega)}+\|R_h\|_C \leq C_R h^{-s} \|r_h\|_C$,
with $s=1/2$ when $C$ is a subset of $\partial \Omega$ and $s=1$ when 
$C$ is a subset of $\Omega$.
\end{lemma}
\begin{proof}
Define $R_h$
so that $R_h(x) = \xi_h r_h(x)$ for all nodes $x$ in $\mathcal{T}_C$ and $R_h(x)=0$ for
all other nodes $x$ in the mesh.
First consider the case when $C$ is a subset of the bulk domain
$\Omega$. Then, using an inverse inequality,
\[
\|\nabla R_h\|_{\Omega} \leq C_i h^{-1} \|r_h\|_{\Omega} = C_i h^{-1} \|r_h\|_{\Omega} = C_i h^{-1} \|r_h\|_{C}. 
\]
In the case $C$ is a subset of the boundary of $\Omega$ we observe that
\[
\|\nabla R_h\|_{\Omega} = \left(\sum_{K \subset \mathcal{T}: \partial
      K \cap C \ne \emptyset}
\|\nabla R_h\|_{K}^2 \right)^{\frac12}  \leq
\left(\sum_{K \subset \mathcal{T}:  \partial K \cap C \ne \emptyset}
h^{-2} \|R_h\|_{K}^2 \right)^{\frac12}.
\]
Using that $R_h$ is defined by the nodes in $C$, combined with the
shape regularity of the mesh, we may use the following inverse trace
inequality \cite[Lemma 3.1]{BE07} on every $K: \partial K \cap C \ne \emptyset$,
\[
\|R_h\|_{K} \leq C h^{\frac12} \|R_h\|_{\partial K \cap C}.
\]
It follows, since $R_h\vert_C = \xi_h r_h$, that
\[
\|\nabla R_h\|_{\Omega} \leq C h^{-1/2} \|R_h\|_{C} \leq  C h^{-1/2} \|r_h\|_{C}.
\]
\end{proof}
\section{Existence of unique discrete solution}
In the previous works on Nitsche's method for contact problems \cite{CH13b,CHR15} existence and uniqueness has
been proven by using the monotonicity and hemi-continuity of the
operator. Here we propose a different approach where we use
Brouwer's fixed point theorem to establish existence and the 
monotonicity of the nonlinearity for uniqueness. To this end we
introduce the finite dimensional nonlinear system corresponding to the
formulation \eqref{FEM}. 

Let $M:= N_V + N_\Lambda$, where $N_V$ and $N_\Lambda$ denote the number of
degrees of freedom of $V_h$ and $\Lambda_h$ respectively. Then define
$U,V \in \mathbb{R}^M$, where $U = \{u_i\}_{i=1}^{N_V} \cup \{\lambda_i\}_{i=1}^{N_\Lambda}$, $V=\{v_i\}_{i=1}^{N_V} \cup \{\mu_i\}_{i=1}^{N_\Lambda}$,
where $\{u_i\}, \{v_i\}$ and $\{\lambda_i\},\{v_i\}$ denote the vectors of unknowns
associated to the basis functions of $V_h$ and $\Lambda_h$ respectively.

Consider the mapping $G:\mathbb{R}^M \mapsto \mathbb{R}^M$ defined by
\[
(G(U),V)_{\mathbb{R}^M}  := A_h[(u_h,\lambda_h),(v_h,\mu_h)] -
(f,v_h)_\Omega.
\]
Existence and uniqueness of a solution to \eqref{FEM} is equivalent to
showing that there exists a unique $U \in \mathbb{R}^M$ such that $G(U)=0$.

We start by showing
some positivity results and a priori bounds
\begin{lemma}\label{lem:apriori}
There exists $\alpha>0$ and an associated constant $c_\alpha>0$ so
that with the form $b$
defined by \eqref{stab_form1}, $\delta>0$ and $\gamma= \gamma_0 h^{2 s}$ with
$\gamma_0 > 0$ there holds, for all $(u_h,\lambda_h) \in V_h \times \Lambda_h$
\begin{multline}\label{form2_stab}
\|\nabla u_h\|_\Omega^2 + 
\gamma \|\lambda_h +  \gamma^{-1} [P_{\gamma+}(u_h,\lambda_h)]_+\|_C^2 +
c_\alpha \|\gamma^{\frac12} \lambda_h\|_C^2 \\
\lesssim A_h[(u_h,\lambda_h),(u_h-\alpha R_h, \lambda_h)],
\end{multline}
where $R_h \in V_h$ is defined in Lemma \ref{lem:lift}, such that
$R_h\vert_C := \gamma \xi_h I_{cf} \lambda_h$.

There exists $\alpha>0$ and an associated constant $c_\alpha>0$ so
that with the form $b$
defined by \eqref{infsup_stabform1}, $k \ge 2$ and $\gamma=
\gamma_0 h^{2 s}$ with $\gamma_0 > 0$, $\gamma_0$ sufficiently large,
and $\delta>0$  there holds, for all $(u_h,\lambda_h) \in V_h \times \Lambda_h$
\begin{multline}\label{form1_stab}
\|\nabla u_h\|_\Omega^2 + 
\gamma^{-1} \|u_h +  [P_{\gamma-}(u_h,\lambda_h)]_+\|_C^2 +
c_\alpha \|\gamma^{\frac12} \lambda_h\|_C^2  \\
\lesssim A_h[(u_h,\lambda_h),(u_h+\alpha R_h,\lambda_h +
\gamma^{-1}\pi_0 u_h)],
\end{multline}
with $R_h$ as before. In
case $k=1$ \eqref{form1_stab} holds under the additional that $0<\delta
\leq (c_0 C_T)^2 \gamma_0^{-1}$.

Under the same conditions on the
parameters as above, for both formulations there also holds, for
$(u_h,\lambda_h)$ solution of \eqref{CompFEM}, 
\begin{equation}\label{form_apriori}
\|\nabla u_h\|_\Omega + \|\gamma^{\frac12} \lambda_h\|_C \lesssim \|f\|_\Omega.
\end{equation}
The hidden constants are independent of $h$.
\end{lemma}
\begin{remark}
For $k\ge 2$ and continuous multiplier space the parameter $\delta$
and the term
$|\lambda_h|_s^2$ can be dropped above.
\end{remark}
\begin{proof}
First consider the claims for formulation 1. By testing in
\eqref{CompFEM} with $v_h = u_h$ and $\mu_h = \lambda_h$ and observing
that
\begin{multline*}
\left<
 \gamma^{-1}[P_{\gamma+}(u_h,\lambda_h)]_+, \gamma  \lambda_h\right>_{C} 
+\|\gamma^{\frac12}
  \lambda_h\|^2_{C} \\=\|\gamma^{\frac12}
  \lambda_h\|^2_{C} + \left<
 \gamma^{-1}[P_{\gamma+}(u_h,\lambda_h)]_+, - \gamma  \lambda_h\right>_{C} 
+ 2 \left<
 \gamma^{-\frac12}[P_{\gamma+}(u_h,\lambda_h)]_+,  \gamma^{\frac12}  \lambda_h\right>_{C} 
\end{multline*}
we obtain
the relation
\begin{equation*}
\|\nabla u_h\|_\Omega^2 + \gamma \|\gamma^{-\frac12}[P_{\gamma+}(u_h,\lambda_h)]_+ +
\lambda_h\|_C^2 + |\lambda_h|_s^2\\
= A_h[(u_h,\lambda_h),(u_h,\lambda_h)] 
\end{equation*}
and hence by using the Cauchy-Schwarz inequality and a Poincar\'e
inequality in the right hand side
\begin{equation}\label{eq:form1_bound1}
\frac12 \|\nabla u_h\|_\Omega^2 + \gamma \|\gamma^{-\frac12}[P_{\gamma+}(u_h,\lambda_h)]_+ +
\lambda_h\|_C^2 + |\lambda_h|_s^2 \lesssim \|f\|^2_\Omega
\end{equation}
Using now the first equation we have testing with $v_h = -\alpha R_h$,
with $r_h = \gamma  I_{cf}(\xi_h \lambda_h)$ and
$\mu_h = 0$,
\begin{align}\nonumber
&A_h[(u_h,\lambda_h),(- R_h,0)] 
\\ \nonumber
&\qquad = a(u_h,- R_h) +
\left<\gamma^{-1}[P_{\gamma+}(u_h,\lambda_h)]_+,-\gamma I_{cf} (\xi_h \lambda_h)\right>_C 
\\ \nonumber
&\qquad = a(u_h,- R_h) +
\left<\gamma^{-1}[P_{\gamma+}(u_h,\lambda_h)]_++ \lambda_h,-\gamma
  I_{cf} (\xi_h \lambda_h)\right>_C 
  \\ \label{eqrepeat}
&\qquad \qquad -\left<\lambda_h,\gamma (\xi_h \lambda_h-
  I_{cf} (\xi_h \lambda_h))\right>_C + \left<\lambda_h,\gamma \xi_h \lambda_h \right>_C.
\end{align}
For the last term in the right hand side we have by the inequality \eqref{eq:xicont},
$
c^2_\xi \|\gamma^{\frac12} \lambda_h\|_C^2 \leq (\gamma \lambda_h, \xi_h \lambda_h)_C$.
The second to last term of the right hand side, which is zero for continuous multiplier spaces, can be bounded
using Proposition \ref{lem:jump_cont}
\begin{equation}\label{eq:bound1}
 (\gamma
\lambda_h, \xi_h \lambda_h - I_{cf}(\xi_h \lambda_h))_C \leq
c^2_\xi \frac14\|\gamma^{\frac12} \lambda_h\|_C^2 + c_s^2 c^{-2}_\xi 
\delta^{-1} |\lambda_h|_s^2.
\end{equation}
The second term is bounded using a Cauchy-Schwarz inequality and the
stability of $I_{cf}$,
\begin{multline}\label{eq:bound2}
\left<\gamma^{-1}[P_{\gamma+}(u_\lambda,\lambda_h)]_++\lambda_h, \gamma I_{cf}(\xi_h \lambda_h)
\right>_C \\
\leq \frac12 (c_{cf} c_\xi )^{-2} \gamma
\|\gamma^{-1}[P_{\gamma+}(u_\lambda,\lambda_h)]_++\lambda_h\|^2_C +
\frac14 c^2_\xi \|\gamma^{\frac12} \lambda_h\|_C^2 
\end{multline}
for the first term we use the Caucy-Schwarz inequality
followed by the stability of $R_h$ and of $I_{cf}$ to obtain
\begin{equation}\label{eq:bound3}
 a(u_h,R_h) \leq  C_R^2 h^{-2s} \gamma c_{cf}^2 c^{-2}_\xi \|\nabla u_h\|_\Omega^2 + c^2_\xi \frac14 \|\gamma^{\frac12} \lambda_h\|_C^2.
\end{equation}
Applying the inequalities \eqref{eq:bound1}-\eqref{eq:bound3} to
\eqref{eqrepeat} we have
\begin{align}\nonumber
& c^2_\xi \alpha \|\gamma^{\frac12} \lambda_h\|_C^2 -  C_R^2 h^{-2s}
\gamma c_{cf}^2 c^{-2}_\xi \alpha \|\nabla u_h\|_\Omega 
\\ \nonumber
&\qquad\qquad -\frac12
(c_{cf} c_\xi )^{-2} \gamma \alpha 
\|\gamma^{-1}[P_{\gamma+}(u_\lambda,\lambda_h)]_+ +\lambda_h\|^2_C
\\   \nonumber
&\qquad\qquad -  c_s^2 c^{-2}_\xi 
\delta^{-1} \alpha |\lambda_h|_s^2 
\\ \label{eq:multfinal}
&\qquad 
\leq A_h[(u_h,\lambda_h),(- \alpha R_h,0)]
\end{align}
We conclude by observing that $ h^{-2s}
\gamma = O(1)$ and by combining the bounds \eqref{eq:xicont},
\eqref{eq:form1_bound1} and \eqref{eq:multfinal} with $\alpha$ small
enough. The a priori estimate follows noting that for
$(u_h,\lambda_h)$ solution of \eqref{FEM} there holds using the
Poincar\'e inequality and the properties of $R_h$,
\[
A_h[(u_h,\lambda_h),(u_h - \alpha R_h,0)] = (f,u_h - \alpha R_h)
\leq C \|f\|_\Omega(\|\nabla u_h\|_{\Omega} + \|\gamma^{\frac12} \lambda_h\|_C).
\]

To prove \eqref{form1_stab} we start by testing in the left hand side of \eqref{CompFEM} with $v_h=u_h$ and $\mu_h =
\lambda_h + \gamma^{-1}\pi_i u_h = \gamma P_{\gamma-}(u_h,\lambda_h) +
\gamma^{-1} (u_h + \pi_i u_h)$, where $i=0$ if $k=1$ and $i=1$ for
$k\ge 2$. 
Observing this time that, using the definition
\eqref{infsup_stabform1} and adding and subtracting $u_h$ at suitable
places the following equality holds
\begin{align*}
b[(u_h,\lambda_h),(u_h,\lambda_h + \gamma^{-1}\pi_i u_h)] 
&= \gamma^{-1}
\left< u_h,u_h\right>_C-\gamma^{-1}\|\pi_i u_h - u_h\|^2_C 
\\
&\qquad + \gamma^{-1}
\|[P_{\gamma-}(u_h,\lambda_h)]_+\|_C^2 
\\
&\qquad 
+ 2 \left<\gamma^{-1} 
  u_h,[P_{\gamma-}(u_h,\lambda_h)]_+\right>_C
\\
&\qquad 
  +\gamma^{-1} \left<\pi_i
  u_h - u_h,[P_{\gamma-}(u_h,\lambda_h)]_+\right>_C 
\\
&\qquad
+  |\lambda_h|_s^2 +s(\lambda_h,  \gamma^{-1}(\pi_i u_h - u_h)).
\end{align*}
This results in 
\begin{multline*}
\|\nabla u_h\|_\Omega^2 + 
\gamma^{-1} \|u_h +  [P_{\gamma-}(u_h,\lambda_h)]_+\|_C^2 + |\lambda_h|_s^2 -\gamma^{-1} \|\pi_i u_h
  - u_h\|^2_C \\
+ \gamma^{-1} \left<\pi_i u_h - u_h, u_h+
  [P_{\gamma-}(u_h,\lambda_h)]_+\right>_C +s(\lambda_h,  \gamma^{-1}(\pi_i u_h - u_h))\\
= A_h[(u_h,\lambda_h),(v_h,\mu_h)].
\end{multline*}
We now bound the three last terms on the left hand side. First by the
properties of $\pi_i$ we have
\begin{equation}\label{eq:bound12}
\gamma^{-1}\|\pi_i u_h - u_h\|^2_C \leq c^2_i h^{2s} \gamma^{-1}
  \|\nabla u_h\|_\Omega^2.
\end{equation}
Using a Cauchy-Schwarz inequality, the previous result and an
arithmetic-geometric inequality we have
\begin{multline}\label{eq:bound22}
\gamma^{-1} \left<\pi_i u_h - u_h, u_h+
  [P_{\gamma-}(u_h,\lambda_h)]_+\right>_C \leq \frac12 c^2_i h^{2s} \gamma^{-1}
  \|\nabla u_h\|_\Omega^2 \\+ \frac12 \gamma^{-1} \|u_h +  [P_{\gamma-}(u_h,\lambda_h)]_+\|_C^2 .
\end{multline}
Finally for $k=1$ we have for the last term
\begin{align*}
s(\lambda_h,  \gamma^{-1}(\pi_0 u_h - u_h)) 
&\leq \frac12
|\lambda_h|_s^2+  \gamma^{-1} \delta C_T^2\|\pi_0 u_h
  - u_h\|_C^2 
\\
&\leq \frac12
|\lambda_h|_s^2+\frac12  \frac{\delta c_0^2 C_T^2 }{\gamma_0} \|\nabla  u_h\|_\Omega^2 
\end{align*}
and for $k\ge 2$, $s(\lambda_h,  \gamma^{-1}(\pi_1 u_h - u_h)) = 0$.
Collecting the results above we obtain for $k=1$
\begin{align} \nonumber
&(1 - 3/2 c_0^2 \gamma_0^{-1}- 1/2 \delta c_0^2 C_T^2\gamma_0^{-1})
\|\nabla u_h\|_\Omega^2 
\\ \nonumber
&\qquad \qquad + 
\frac12 \gamma^{-1} \|u_h +  [P_{\gamma-}(u_h,\lambda_h)]_+\|_C^2 + \frac12
|\lambda_h|_s^2  
\\ \label{eq:1s_apriori}
&\qquad \lesssim  A_h[(u_h,\lambda_h),(v_h,\mu_h)].
\end{align}
We see that the factor $(1 - 3/2 c_0^2 \gamma_0^{-1}- 1/2 \delta c_0^2
C_T^2\gamma_0^{-1})$ is positive under the assumptions on $\gamma_0$
and $\delta$.
The corresponding inequality for $k\ge 2$ is obtained by omitting the
term with $\delta$ and replacing $c_0$ with $c_1$.
Observe that by using $v_h = R_h$ with $r_h = - \gamma I_{cf} \xi_h \lambda_h$ and
$\mu_h = 0$ we have using similar arguments as above
\begin{equation}\label{eq:1st_lam}
\gamma\|\xi_h^{\frac12} \lambda_h\|_C^2 + \left<\lambda_h, \gamma (\xi_h
  \lambda_h -  I_{cf} \xi_h \lambda_h) \right>_C  + a(u_h,R_h) = A_h[(u_h,\lambda_h),(R_h,0)].
\end{equation}
Using once again \eqref{eq:bound1} and \eqref{eq:bound3} 
\begin{align}\nonumber
&\frac12 c_\xi \| \gamma^{\frac12}\lambda_h\|_C - C_R^2 h^{-2s} \gamma
c_{cf}^2 c_\xi^{-2} \|\nabla u_h\|^2_\Omega - c_s^2 c_\xi^{-2} \delta^{-1}
|\lambda_h|^2_s 
\\ \label{lambda_apriori}
&\qquad 
\leq
 A_h[(u_h,\lambda_h),(R_h,0)]
\end{align}
where the stabilization contribution can be dropped whenever 
continuous approximation is used for the multiplier space. We conclude
as in the previous case by combining the bounds \eqref{lambda_apriori}
and \eqref{eq:1s_apriori}. The a priori estimate \eqref{form_apriori} also follows as before.
\end{proof}
\begin{proposition}
The formulation \eqref{CompFEM} using the contact operators \eqref{infsup_stabform1} or
\eqref{stab_form1}, and the same assumptions on the parameters
$\delta$, $\gamma$ as in Lemma \ref{lem:apriori}, admits a unique solution.
\end{proposition}
\begin{proof}
 By the
positivity results \eqref{form2_stab} and \eqref{form1_stab} of Lemma \ref{lem:apriori} we have for each method that there exists a
linear mapping $B:\mathbb{R}^M \mapsto \mathbb{R}^M$ such that $b_1
|U| <|B U|
\leq b_2 |U|$ for some $0<b_1\leq b_2$ and that for $U$ sufficiently
big
\begin{equation}\label{positive_G}
0 < (G(U),B U).
\end{equation}

We give details regarding the construction of $B$ only in the case of
formulation 2 with $k=1$.
The argument for $k \ge 2$, and that for formulation 1, are similar.
Let the positive constants $c_h$ and $C_h$ denote the smallest and the
largest eigenvalues respectively of the block diagonal matrix in
$\mathbb{R}^{M\times M}$ with
diagonal blocks given by given by $ (\nabla \varphi_i,\nabla
\varphi_j)_{\Omega}+ \left<\gamma^{-1}\varphi_i,\varphi_j  \right>_C$, $1\leq i,j\leq
N_V$ where $\varphi_i$, denotes the basis functions for
the space $V_h$ and $\tfrac12 \gamma  (\psi_i,\psi_j)_C$ where  $\psi_i$, denotes the basis functions for
the space $\Lambda_h$, $1\leq i,j\leq
N_\Lambda$  such that, with $ \| u_h\|^2_{1,h}:= \|\nabla u_h\|^2_{\Omega}+ \|\gamma^{-\frac12} u_h\|^2_{C}$
\[
c_h |U|^2_{\mathbb{R}^{M}} \leq  \| u_h\|^2_{1,h}+
\frac14 \gamma  \|\lambda_h\|^2_C \leq C_h |U|^2_{\mathbb{R}^{M}}.
\]
Recalling the a priori bound \eqref{form1_stab}, let $B$ denote the transformation matrix such that the finite element function
corresponding to the vector $B U$ is the function $(u_h+ \alpha R_h,
\lambda_h +  \gamma^{-1}\pi_0 u_h)$, with $R_h$ defined in Lemma
\ref{lem:apriori}. First we show that for $\alpha$ sufficiently small, there are constants $b_1$ and $b_2$ such that $b_1 |U|_{\mathbb{R}^{M}}
\leq |BU|_{\mathbb{R}^{N_V}} \leq b_2 |U|_{\mathbb{R}^{M}}$. This can be seen by observing that
\begin{align*}
\| u_h\|^2_{1,h}+ \frac14 \gamma \|\lambda_h\|^2_C 
&\leq 
2 \| u_h+\alpha R_h\|^2_{1,h}+  \frac12 \gamma
\|\lambda_h + \gamma^{-1} \pi_0 u_h\|^2_C
\\ 
&\qquad + \frac12\|\gamma^{-1/2} \pi_0 u_h\|^2_C +
2 \| \alpha R_h\|^2_{1,h}
\\
&\leq 
2 C_h |B U|^2_{\mathbb{R}^{M}} + \frac12 \|u_h\|^2_{1,h} + C \alpha \gamma \|\lambda_h\|^2_C
\end{align*}
where we have used the properties of $R_h$ from Lemma \ref{lem:lift}.
It follows, for $\alpha$ small enough, that
$$
\frac12 c_h  |U|^2_{\mathbb{R}^{M}}\leq
(1 -\frac12)\|u_h\|^2_{1,h}+ (\frac14 - C \alpha)
\gamma\|\lambda_h\|^2_C\leq 2 C_h |BU|^2_{\mathbb{R}^{M}}.
$$

Similarly we may prove the upper bound using that by the properties of
$R_h$ and $\pi_0 u_h$ we have
\begin{align*}
c_h |B U|^2_{\mathbb{R}^{M}} 
&\leq 
\|u_h+\alpha R_h\|^2_{1,h}+\frac14 \gamma \|\lambda_h + \gamma^{-1} \pi_0 u_h\|^2_{C}\\
&\leq 
2 \|u_h\|^2_{1,h}+2 \|\alpha R_h\|^2_{1,h}+\frac14 \gamma \|\lambda_h\|^2_{C} + \frac12  \|\gamma^{-1/2} \pi_0 u_h\|^2_{C}
\\
&\leq C (\|u_h\|^2_{1,h}+\frac14 \gamma \|\lambda_h\|^2_{C}).
\end{align*}

Assume that the positivity \eqref{positive_G} holds whenever $|U|_{\mathbb{R}^{M}}\ge q$.
Assume now that there is no $U$ such that $G(U)=0$ and define the
function
\[
\phi(U) = - q/b_1 B^T G(U)/|G(U)|_{\mathbb{R}^{M}}.
\]
Then $\phi:B_{R} \mapsto B_{R}$, where $R= q b_2/b_1$ and, since $G(U)
\ne 0$, $\phi$ is
continuous by Lemma \ref{bcont} and equivalence of norms on finite
dimensional spaces. Hence, since $B^T$ satisfies the same
bounds as $B$, there
exists a fixed point $X \in B_R$, with $$|X|_{\mathbb{R}^M} = q/b_1|
B^T G(X)|_{\mathbb{R}^{M}}/|G(X)|_{\mathbb{R}^{M}}\ge q $$ such that 
\[
X = \phi(X).
\]
It follows that 
\[
0<|X|^2_{\mathbb{R}^M} = -q/b_1 (G(X),B X)_{\mathbb{R}^M}/|G(X)|_{\mathbb{R}^{M}}
\]
but by assumption $(G(X),B X)>0$ for $|X|_{\mathbb{R}^M}\ge q$, which leads to a
contradiction. It follows that the finite dimensional nonlinear system
admits at least one solution.

Uniqueness is consequence of the positivity results of Lemma
\ref{lem:apriori} and the monotonicity of Lemma \ref{lem:monotone}.
Considering first formulation 1, where the form $b[\cdot;\cdot]$ is
given by \eqref{stab_form1}, we have
\begin{align*}
&\|\nabla (u_1-u_2)\|_\Omega^2 
\\
&\qquad = -\gamma^{-1}
\left<[P_{\gamma+}(u_1,\lambda_1)]_+-[P_{\gamma+}(u_2,\lambda_2)]_+,
 u_1-u_2+\gamma(\lambda_1-\lambda_2)\right>_C 
\\
&\qquad \qquad - \gamma \|\lambda_1-\lambda_2\|_C^2 - |\lambda_1-\lambda_2|_s^2.
\end{align*}
It follows that, defining
\[
|||u,\lambda|||^2:= \|\nabla u\|_\Omega^2 + |\lambda|_s^2 ,
\]
\begin{align*}
&|||u_1-u_2,\lambda_1-\lambda_2|||^2 
\\
&\qquad 
= - \gamma \|\lambda_1-\lambda_2\|_C^2
\\
&\qquad \qquad 
-\gamma^{-1}\left<[P_{\gamma+}(u_1,\lambda_1)]_+-[P_{\gamma+}(u_2,\lambda_2)]_+,
  P_{\gamma+}(u_1-u_2,\lambda_1-\lambda_2)\right>_C
\\
&\qquad \qquad 
+ \gamma^{-1}\left<[P_{\gamma+}(u_1,\lambda_1)]_+-[P_{\gamma+}(u_2,\lambda_2)]_+,
  2 \gamma(\lambda_1-\lambda_2)\right>_C.
\end{align*}
Then, using the monotonicity of Lemma \ref{lem:monotone} we deduce
\begin{align*}
&|||u_1-u_2,\lambda_1-\lambda_2|||^2 + \gamma \|\lambda_1-\lambda_2\|_C^2
\\ 
&\qquad \qquad + \gamma^{-1}
\|[P_{\gamma+}(u_1,\lambda_1)]_+-[P_{\gamma+}(u_2,\lambda_2)]_+\|^2_C 
\\
&\qquad \leq - \left<[P_{\gamma+}(u_1,\lambda_1)]_+-[P_{\gamma+}(u_2,\lambda_2)]_+,
  2 \gamma(\lambda_1-\lambda_2)\right>_C.
\end{align*}
Therefore
\begin{equation}\label{part_unique}
|||u_1-u_2,\lambda_1-\lambda_2|||^2 + \gamma^{-1}
\|(\lambda_1-\lambda_2)
+[P_{\gamma+}(u_1,\lambda_1)]_+-[P_{\gamma+}(u_2,\lambda_2)]_+\|^2_C = 0
\end{equation}
and $u_1 = u_2$. Repeating the arguments leading to \eqref{eq:multfinal} on
$\lambda_1-\lambda_2$ and using \eqref{part_unique} allows us to conclude that
$\lambda_1=\lambda_2$.

In the case of formulation 2
we only give the details for $k\ge 2$, the case $k = 1$ is similar,
but we need to handle an additional stabilization term. Assume that $(u_1,\lambda_1)$ and
$(u_2,\lambda_2)$ solves \eqref{CompFEM} with the contact
conditions defined by \eqref{infsup_stabform1}. 
\begin{align*}
\|\nabla (u_1-u_2)\|_\Omega^2 = {}& \left<\lambda_1 - \lambda_2,u_1-u_2
\right>_C\\
= {}& -\gamma \left<\lambda_1 - \lambda_2,[P_{\gamma-} (u_1,\lambda_1)]_+ -
  [P_{\gamma-} (u_2,\lambda_2)]_+
\right>_C \\
{}&  - |\lambda_1 - \lambda_2|^2_s.
\end{align*}
Observing that with $\mu_h = \gamma^{-1}\pi_1 (u_1 - u_2)$ we also
have
\begin{multline*}
\gamma^{-1} \|\pi_1 (u_1 - u_2)\|_C^2 + \gamma^{-1} \left<\pi_1 (u_1 -
u_2), [P_{\gamma-} (u_1,\lambda_1)]_+ -
  [P_{\gamma-} (u_2,\lambda_2)]_+
\right>_C = 0
\end{multline*}
and therefore we can write
\begin{multline*}
\|\nabla (u_1-u_2)\|_\Omega^2 + \gamma^{-1} \|u_1-u_2 +  [P_{\gamma-} (u_1,\lambda_1)]_+ -
  [P_{\gamma-} (u_2,\lambda_2)]_+\|_C^2 +  |\lambda_1 - \lambda_2|^2_s\\
  = \gamma^{-1} \left<(1-\pi_1) (u_1 - u_2), u_1-u_2 +  [P_{\gamma-} (u_1,\lambda_1)]_+ -
  [P_{\gamma-} (u_2,\lambda_2)]_+\right>_C.
\end{multline*}
By splitting the term in the right hand side using the
arithmetic-geometric inequality and using the approximation properties
of $\pi_1$, $$ \|(1-\pi_1)
(u_1 - u_2)\|_C \leq c_1 h^{s} \|\nabla
(u_1-u_2)\|_\Omega$$ we may conclude that
\begin{align*}
&(1 - \gamma^{-1} c_1^2 h^{2s}) \|\nabla (u_1-u_2)\|_\Omega^2 
\\
&\qquad \qquad + \frac12
\gamma^{-1} \|u_1-u_2 +  [P_{\gamma-} (u_1,\lambda_1)]_+ -
  [P_{\gamma-} (u_2,\lambda_2)]_+\|_C^2 
\\
&\qquad \qquad
+  |\lambda_1 -
  \lambda_2|^2_s
\\  
&\qquad  \leq 0.
\end{align*}
As a consequence $u_1=u_2$ when $\gamma_0$ is sufficiently large. That $\lambda_1=\lambda_2$ is immediate
from \eqref{eq:1st_lam} since the first equation of \eqref{eq:stat_point2} is linear.
\end{proof}

\section{Error estimates}
In this section we will prove the main results of the paper which are
error estimates for the two methods given by \eqref{FEM}
with the two contact formulations \eqref{infsup_stabform1} and
\eqref{stab_form1}. The idea of the proof is to combine the uniqueness
argument with a Galerkin type perturbation analysis. Since this result
is central to the present work we give full detail for both
formulations. 
\begin{theorem}\label{apriori_error_1}(Formulation 1)
Assume that $u \in H^1(\Omega)$ and $\lambda \in L^2(C)$ is the unique
stationary point of \eqref{functional} and $(u_h, \lambda_h)$ the solution to \eqref{FEM}
with \eqref{stab_form0} and $0<\gamma =\gamma_0 h^{2s}$, where $s=1/2$ for the Signorini problem and $s=1$
for the Obstacle problem, $\gamma_0 \in\mathbb{R}^+$ is sufficiently small and $\delta
\in\mathbb{R}^+$ sufficiently large.
 Then there holds for all $(v_h, \mu_h) \in
V_h \times \Lambda_h$
\begin{multline*}
\alpha \|u-u_h\|_{H^1(\Omega)}^2  + \gamma\| (\lambda -
  \lambda_h)\|_{C}^2 + \gamma \|(\lambda + \gamma^{-1}
[P_{\gamma+}(u_h,\lambda_h)]_+)\|^2_{C} \\
\lesssim \frac{1}{\alpha}
\|u-v_h\|_{H^1(\Omega)}^2 + \gamma\| (\lambda -
  \mu_h)\|_{C}^2 + \gamma^{-1}\| (u - v_h)\|^2_{C} + s(\mu_h,\mu_h).
\end{multline*}
\end{theorem}
\begin{proof}
Using the coercivity of $a(\cdot,\cdot)$ we may write
\begin{align*}
\alpha \|u-u_h\|_{H^1(\Omega)}^2 &\leq  a(u-u_h,u-u_h) 
\\
&= a(u-u_h,u -
v_h) +a(u-u_h,v_h-u_h) 
\\
&\leq \frac{\alpha}{4} \|u-u_h\|_{H^1(\Omega)}^2
+ \frac{1}{\alpha} \|u-v_h\|_{H^1(\Omega)}^2 +a(u-u_h,v_h-u_h).
\end{align*}
It follows, using Galerkin orthogonality, that
\begin{align}\nonumber
&a(u-u_h,v_h-u_h) 
\\ \nonumber
&\qquad =  \left<\gamma^{-1}
  [P_{\gamma+}(u_h,\lambda_h)]_+-\gamma^{-1} [P_{\gamma+}(u,\lambda)]_+,
  v_h-u_h \right>_{C}
  \\ \nonumber
&\qquad = \left<\gamma^{-1}
  [P_{\gamma+}(u_h,\lambda_h)]_+-\gamma^{-1} [P_{\gamma+}(u,\lambda)]_+,
  v_h-u_h+\gamma(\mu_h-\lambda_h) \right>_{C} 
  \\ \label{eq:GO_pert}
&\qquad \qquad +
\left<\gamma (\lambda_h - \lambda), (\mu_h - \lambda_h) \right>  
+ s(\lambda_h,\mu_h - \lambda_h).
\end{align}
First observe that
\begin{align*}
\left<\gamma (\lambda_h - \lambda), (\mu_h - \lambda_h) \right> = {}& -
\|\gamma^{\frac12}(\mu_h - \lambda_h )\|^2_C \\
{}& + \|\gamma^{\frac12}
(\mu_h - \lambda)\|_C\|\gamma^{\frac12} (\mu_h - \lambda_h)\|_C \\
\leq {}& \left(\varepsilon_1 -1\right) \|\gamma^{\frac12}(\mu_h -
\lambda_h )\|^2_C + \frac{1}{4\varepsilon_1} \|\gamma^{\frac12}
(\mu_h - \lambda)\|_C^2
\end{align*}
where we see that the first term can be made negative by choosing
$\varepsilon_1$ small enough. Similarly
\[
s(\lambda_h,\mu_h - \lambda_h) = -
|\mu_h -\lambda_h|^2_s + s(\mu_h,\mu_h - \lambda_h)
\leq (\varepsilon_2 -1) |\mu_h -\lambda_h|^2_s + \frac{1}{4\varepsilon_2} s(\mu_h,\mu_h)
\]
where once again the first term on the right hand side can be made
negative by choosing $\varepsilon$ small.
Considering the first term on the right hand side of equation
\eqref{eq:GO_pert} we may write
\begin{align*}
&\left<\gamma^{-1}
  [P_{\gamma+}(u_h,\lambda_h)]_+-\gamma^{-1} [P_{\gamma+}(u,\lambda)]_+,
   v_h-u_h+\gamma(\mu_h-\lambda_h) \right>_{C} 
   \\
&\qquad 
= \underbrace{\left<\gamma^{-1}
  [P_{\gamma+}(u_h,\lambda_h)]_+-\gamma^{-1} [P_{\gamma+}(u,\lambda)]_+,
  P_{\gamma+}(v_h-u,\mu_h-\lambda) \right>_{C} }_{I}
  \\
&\qquad \qquad 
+ \underbrace{\left<\gamma^{-1}
  [P_{\gamma+}(u_h,\lambda_h)]_+-\gamma^{-1} [P_{\gamma+}(u,\lambda)]_+,
  P_{\gamma+}(u-u_h,\lambda-\lambda_h) \right>_{C} }_{II} 
  \\
&\qquad \qquad
+ \underbrace{\left<\gamma^{-1}
  [P_{\gamma+}(u_h,\lambda_h)]_+-\gamma^{-1} [P_{\gamma+}(u,\lambda)]_+,2
  \gamma(\mu_h-\lambda_h) \right>_{C}}_{III}
\\
&\qquad = I+II+III
\end{align*}
The term $I$ may be bounded using the Cauchy-Schwarz inequality
followed by the arithmetic geometric inequality
\[
I \leq \varepsilon_3 \|\lambda+ \gamma^{-\frac12}
[P_{\gamma+}(u_h,\lambda_h)]_+\|_{C}^2 + \frac{1}{4 \varepsilon_3}
\|\gamma^{-\frac12} P_{\gamma+}(v_h-u,\mu_h-\lambda)\|_{C}^2.
\]
For the term $II$ we use the monotonicity property $([a]_+-[b]_+)(b-a)
\leq -([a]_+-[b]_+)^2$ to deduce that
\[
II \leq -\|\gamma^{\frac12}(\lambda + \gamma^{-1}
[P_{\gamma+}(u_h,\lambda_h)]_+)\|^2_C
\]
Finally to estimate term $III$, let $R_h$ be defined by Lemma \ref{lem:lift} with the associated
$r_h :=I_{cf} (2
 \xi_h \gamma(\mu_h-\lambda_h))$ and set $\zeta_h = 1-\xi_h$.
Using the equation we may write
\begin{align*}
&\left<\gamma^{-1}
  [P_{\gamma+}(u_h,\lambda_h)]_+-\gamma^{-1} [P_{\gamma+}(u,\lambda)]_+,2
  \gamma(\mu_h-\lambda_h) \right>_{C} 
\\
&\qquad 
\leq \underbrace{\left<\gamma^{-1}
  [P_{\gamma+}(u_h,\lambda_h)]_+-\gamma^{-1} [P_{\gamma+}(u,\lambda)]_+,2
  \zeta_h \gamma(\mu_h-\lambda_h) \right>_{\Gamma_C \cap G_h}}_{IIIa} 
  \\
&\qquad \qquad +  \underbrace{\left<\gamma^{-1}
  [P_{\gamma+}(u_h,\lambda_h)]_+-\gamma^{-1} [P_{\gamma+}(u,\lambda)]_+,2
  \xi_h \gamma(\mu_h-\lambda_h) -r_h\right>_{\Gamma_C }}_{IIIb}
  \\
&\qquad \qquad + \underbrace{a(u-u_h, R_h)}_{IIIc} 
\\
&\qquad = IIIa+IIIb+IIIc.
\end{align*}
We estimate $IIIa$-$IIIc$ term by term. For $IIIa$ we use the assumption \ref{ass:infsup}
\[
IIIa \leq c_D \|\gamma^{-\frac12} (
[P_{\gamma+}(u_h,\lambda_h)]_++\lambda)\|_{C}^2 + c_D \|\gamma^{\frac12} (\mu_h-\lambda_h)\|_C^2
\]
As a consequence of Lemma \ref{lem:jump_cont} we get the following
bound of term $II$
\begin{equation*}
IIIb \leq
\varepsilon_4 \|\gamma^{-\frac12}(\lambda +
[P_{\gamma+}(u_h,\lambda_h)]_+) \|^2_C+
\frac{c_s^2 \gamma}{\varepsilon_4}  \|h^{\frac12}\jump{\mu_h -
  \lambda_h}\|_{\mathcal{F}}^2.
%\leq  \varepsilon_2 \|\gamma^{\frac12}(\lambda +
%[P_{\gamma+}(u_h,\lambda_h)]_+) \|^2_C +\frac{C_s}{\varepsilon_2 \delta} s(\mu_h -
%  \lambda_h,\mu_h -
%  \lambda_h)
\end{equation*}
%where $\delta>0$ is a stabilisation parameter that may be chosen.
For the third term we observe that by the continuity of $a$ and
Lemma \ref{lem:lift} we have
\begin{align*}
IIIc \leq {}& \|u-u_h\|_{H^1(\Omega)} \|R_h\|_{H^1(\Omega)} \leq C 
\|u-u_h\|_{H^1(\Omega)} h^{-s} \|r_h\|_C \\
\leq {}& C 
\|u-u_h\|_{H^1(\Omega)} h^{-s} \gamma^{\frac12} \|\gamma^{\frac12}(
\lambda_h - \mu_h)\|_C \\
\leq {}& \frac{\alpha}{4} \|u-u_h\|_{H^1(\Omega)}^2 +  C^2 h^{-2s}
  \gamma \alpha^{-1}\|\gamma^{\frac12}(
\lambda_h - \mu_h)\|_C^2.
\end{align*}
Collecting the above bounds and recalling that by definition
$$s(\lambda_h,\lambda_h) = \delta \gamma\|h^{\frac12} \jump{\lambda_h}\|^2_{\mathcal{F}_C},$$ we have
\begin{align*}
&\frac{\alpha}{2}\|u-u_h\|_{H^1(\Omega)}^2 
+ (1-\varepsilon_3-\varepsilon_4-c_D) \|\gamma^{-\frac12}(\lambda +
[P_{\gamma+}(u_h,\lambda_h)]_+) \|^2_C 
\\
&\qquad \qquad + (1-\varepsilon_1 - c_D- C^2
\gamma_0/\alpha)  \|\gamma^{\frac12}
(\mu_h - \lambda_h)\|_C^2 
\\
&\qquad \qquad 
+ (1-\varepsilon_2- c_s^2/(\delta \varepsilon_4)) |\mu_h -\lambda_h|^2_s
\\
&\qquad 
\leq \frac{1}{\alpha} \|u-v_h\|_{H^1(\Omega)}^2 
+  \frac{1}{4 \varepsilon_3}
\|\gamma^{-\frac12} P_{\gamma+}(v_h-u,\mu_h-\lambda)\|_{C}^2
\\
&\qquad \qquad +  \frac{1}{4\varepsilon_1} \|\gamma^{\frac12}
(\mu_h - \lambda)\|_C^2+\frac{1}{4\varepsilon_2} s(\mu_h,\mu_h)
\end{align*}
Observe that as usual when a continuous multiplier space is used all
terms and coefficients associated to the jump operator may be omitted.

Fixing $\varepsilon_1,\varepsilon_3, \varepsilon_4$ and $\gamma_0$
sufficiently small so that
$$\varepsilon_1 + C^2
\gamma_0 /\alpha = \varepsilon_3+\varepsilon_4=(1-c_D)/2,$$ and
$\varepsilon_2$ sufficiently small and $\delta$ sufficiently large so
that $\varepsilon_2 + c_s^2/(\delta \varepsilon_4)<1$, then there holds
\begin{align*}
&\alpha\|u-u_h\|_{H^1(\Omega)}^2 
+\|\gamma^{-\frac12}(\lambda +
[P_{\gamma+}(u_h,\lambda_h)]_+) \|^2_C 
\\
&\qquad \qquad 
+ \|\gamma^{\frac12} (\mu_h - \lambda_h)\|_C^2 + |\mu_h -\lambda_h|^2_s
\\
&\qquad 
\lesssim \frac{1}{ \alpha} \|u-v_h\|_{H^1(\Omega)}^2+  
\|\gamma^{-\frac12} P_{\gamma+}(v_h-u,\mu_h-\lambda)\|_{C}^2
\\
&\qquad \qquad +  \|\gamma^{\frac12}
(\mu_h - \lambda)\|_C^2+ s(\mu_h,\mu_h).
\end{align*}
The triangle inequality $\|\gamma^{\frac12}
(\lambda - \lambda_h)\|_C^2 \leq \|\gamma^{\frac12}
(\mu_h - \lambda_h)\|_C^2 +\|\gamma^{\frac12}
(\mu_h - \lambda)\|_C^2 $ concludes the proof.
\end{proof}
\begin{cor}\label{cor:error_est}
Assume that $u \in H^{r}(\Omega)$, $3/2< r \leq k+1$ and $\lambda \in
H^{r-1 - s}(\OmegaC)$, with $r-1 - s>0$ where $s=1/2$ for the Signorini problem and $s=1$
for the Obstacle problem and that $(u_h,\lambda_h)$ is the solution of
\eqref{FEM} with the contact operator defined by \eqref{stab_form0} and under the same conditions on the parameters as in Theorem
\ref{apriori_error_1}. Then there holds
\begin{align*}
&\alpha \|u-u_h\|_{H^1(\Omega)}  +\gamma^{\frac12}\| (\lambda -
  \lambda_h)\|_{C} + \gamma^{\frac12} \|(\lambda + \gamma^{-1}
[P_{\gamma+}(u_h,\lambda_h)]_+)\|_{C} 
\\
&\qquad \lesssim h^{r-1} (|u|_{H^{r}(\Omega)} + |\lambda|_{H^{r-1 - s}(\OmegaC)}),
\end{align*}
\end{cor}
\begin{proof}
Let $v_h = i_h u$ where $i_h$ denotes the standard nodal interpolant
and let $\mu_h= \pi_l \lambda$ where $\pi_l$ denotes the % standard nodal interpolant
$L^2$-projection.
Using standard approximation estimates and the trace inequality
\eqref{trace}  we may then bound the right
hand side of the estimate of Theorem \ref{apriori_error_1},
\[
\|u-v_h\|_{H^1(\Omega)}^2 \lesssim h^{2(r-1)} |u|^2_{H^{r}(\Omega)},
\]
\[
\gamma^{-1}\| (u - v_h)\|^2_{C} \lesssim \gamma^{-1} h^{2(r+s)}
|u|^2_{H^{r}(\Omega)} \lesssim  h^{2(r-1)} |u|^2_{H^{r}(\Omega)},
\]
\[
\gamma\| (\lambda - \mu_h)\|_{C}^2 \lesssim \gamma h^{2(r-1-s)}
|\lambda|_{H^{r-1 - s}(\OmegaC)} \lesssim  h^{2(r-1)}
|\lambda|_{H^{r-1 - s}(\OmegaC)}.
\]
Finally we have, 
\begin{align*}
 s(\mu_h,\mu_h) = {}& s(\pi_l \lambda - \mu_h,\pi_l \lambda - \mu_h)
 \lesssim \gamma h^{-1} \|\pi_l \lambda - \mu_h\|^2_C \\ \lesssim {}& \gamma 
 (\|\pi_l \lambda - \lambda\|^2 + \| \lambda - \mu_h\|^2_C) \lesssim h^{2
   s + 2(r-1 - s)} |\lambda|^2_{H^{r-1 - s}(\OmegaC)} 
\\
\lesssim {}& h^{2(r-1)} |\lambda|^2_{H^{r-1 - s}(\OmegaC)} 
\end{align*}
and we conclude by taking square roots.
\end{proof}
\begin{theorem}\label{apriori_error_2} (Formulation 2)
Assume that $u \in H^1(\Omega)$ and $\lambda \in L^2(C)$ is the unique
stationary point of \eqref{functional} and $(u_h, \lambda_h)$ the solution to \eqref{CompFEM}
with \eqref{infsup_stabform1} and $\gamma = \gamma_0
  h^{2s}$, where $s=1/2$ for the Signorini problem and $s=1$
for the Obstacle problem, $\gamma_0$ sufficiently large and $\delta>0$ then there holds
for all $(v_h , \mu_h) \in
V_h \times \Lambda_h$
\begin{align}\nonumber
&\alpha \|u-u_h\|_{H^1(\Omega)}^2 + \gamma \|(\lambda - \lambda_h)\|_{C}^2  
\\ \nonumber
&\qquad\qquad +\gamma^{-1}\|(u-u_h)+[P_{\gamma-}(u,\lambda)]_+-[P_{\gamma-}(u_h,\lambda_h)]_+\|^2_{C}
\\ \label{lestimate}
&\qquad 
\lesssim  \frac{1}{\alpha} \|u-v_h\|_{H^1(\Omega)}^2
+\
\gamma\|\mu_h-\lambda\|^2_{C}+
|\mu_h|_s^2 + \gamma^{-1} \|v_h-u\|^2_{C}.
\end{align}
\end{theorem}
\begin{proof}
Using the coercivity of $a(\cdot,\cdot)$ we may write
\begin{align}\nonumber
\alpha \|u-u_h\|_{H^1(\Omega)}^2 \leq {}& a(u-u_h,u-u_h) = a(u-u_h,u -
v_h) +a(u-u_h,v_h-u_h) \\ \nonumber
 \leq {}& \frac{\alpha}{4} \|u-u_h\|_{H^1(\Omega)}^2
+ \frac{1}{\alpha} \|u-v_h\|_{H^1(\Omega)}^2\\ {}& +a(u-u_h,v_h-u_h).\label{1st_step}
\end{align}
By Galerkin orthogonality we obtain the equality, and then adding and
subtracting suitable quantities it follows that
\begin{align} \nonumber
a(u-u_h,v_h-u_h) = {}& \left<\lambda-\lambda_h,v_h-u_h\right>_{C}
\\ \nonumber
= {}& \left<\lambda-\lambda_h,v_h-u_h\right>_{C} \\ \nonumber {}& -
\left<\mu_h-\lambda_h,u-u_h\right>_{C}+ s(\lambda_h,\mu_h-\lambda_h) \\ \nonumber {}& -
\left<\mu_h-\lambda_h,[P_{\gamma-}(u,\lambda)]_+-[P_{\gamma-}(u_h,\lambda_h)]_+\right>_{C}.
\end{align}
Then we proceed by adding and subtracting $u$ in the right slot of
the first term on the right hand side, $\lambda$ in the left slot
of the second term, $\mu_h$ in the left slot of the third term and
finally $\lambda+\gamma^{-1}(u-u_h)$ in the left slot of the third
term, leading to
\begin{align} \nonumber
a(u-u_h,v_h-u_h) = {}& \left<\lambda-\lambda_h,v_h-u\right>_{C}
-\left<\mu_h-\lambda,u-u_h\right>_{C} \\ \nonumber
{}& -
\left<\mu_h-\lambda,[P_{\gamma-}(u,\lambda)]_+-[P_{\gamma-}(u_h,\lambda_h)]_+\right>_{C}\\ \nonumber
{}&  -
\gamma^{-1}\left<P_{\gamma-}(u,\lambda)-P_{\gamma-}(u_h,\lambda_h),[P_{\gamma-}(u,\lambda)]_+-[P_{\gamma-}(u_h,\lambda_h)]_+\right>_{C} \\ \nonumber
%{}&+ \left<\gamma^{-1}(u-u_h),[P_{\gamma-}(u,\lambda)]_+-[P_{\gamma-}(u_h,\lambda_h)]_+\right>_{C}\\ \nonumber
{}& - \left<\gamma^{-1}(u-u_h),[P_{\gamma-}(u,\lambda)]_+-[P_{\gamma-}(u_h,\lambda_h)]_+\right>_{C}\\ \nonumber
{}& - |\mu_h - \lambda_h|_s^2+  s(\mu_h,\mu_h-\lambda_h).\\ \nonumber
\end{align}
We may then apply the monotonicity of Lemma \ref{lem:monotone} to
obtain the bound
\begin{align} \nonumber
a(u-u_h,v_h-u_h) \leq {}& \left<\lambda-\lambda_h,v_h-u\right>_{C}\\ \nonumber
{}& -
\left<\mu_h-\lambda,(u - u_h) +[P_{\gamma-}(u,\lambda)]_+-[P_{\gamma-}(u_h,\lambda_h)]_+\right>_{C}\\ \nonumber
{}&  -
\gamma^{-1} \|[P_{\gamma-}(u,\lambda)]_+-[P_{\gamma-}(u_h,\lambda_h)]_+\|^2_{C}\\ \nonumber
{}& - \gamma^{-1}\left<(u-u_h),[P_{\gamma-}(u,\lambda)]_+-[P_{\gamma-}(u_h,\lambda_h)]_+\right>_{C}\\
{}&  - |\mu_h - \lambda_h|_s^2+  s(\mu_h,\mu_h-\lambda_h).\label{eq:GO_pert0}
\end{align}
Summarizing \eqref{1st_step} and \eqref{eq:GO_pert0} we have
\begin{align}\nonumber
&\frac34 \alpha \|u-u_h\|_{H^1(\Omega)}^2 + \frac34 |\mu_h - \lambda_h|_s^2 
+
\gamma^{-1}
\|[P_{\gamma-}(u,\lambda)]_+-[P_{\gamma-}(u_h,\lambda_h)]_+\|^2_{C}
\\ \nonumber
&\qquad \qquad 
+\gamma^{-1}\left<(u-u_h),[P_{\gamma-}(u,\lambda)]_+-[P_{\gamma-}(u_h,\lambda_h)]_+\right>_{C}
\\ \nonumber
&\qquad \leq 
 \frac{1}{\alpha} \|u-v_h\|_{H^1(\Omega)}^2 + \left<\lambda-\lambda_h,v_h-u\right>_{C}
 \\ \label{2nd_step}
&\qquad \qquad  -
\left<\mu_h-\lambda,(u - u_h) +[P_{\gamma-}(u,\lambda)]_+-[P_{\gamma-}(u_h,\lambda_h)]_+\right>_{C}
+  |\mu_h|_s^2.
\end{align}
Then observe that using the second equation we may write, with $\bar e =
  \pi_i (u - u_h)$, with $i=0$ for $k=0$ and $i=1$ for $k\ge 2$ and taking $\mu_h = \gamma^{-1} \bar e$
\begin{align*}
& \gamma^{-1}\left<\bar
   e,[P_{\gamma-}(u,\lambda)]_+-[P_{\gamma-}(u_h,\lambda_h)]_+\right>_{C}
 + \gamma^{-1} \|\bar e\|^2_{C} 
 \\
 &\qquad \qquad 
- \frac14  |\lambda_h|_s^2 - C \gamma^{-1} \delta h^{2s} \|\nabla (u -
 u_h)\|_{\Omega}^2
 \\
&\qquad 
\leq  \gamma^{-1}\left<\bar
   e,[P_{\gamma-}(u,\lambda)]_+-[P_{\gamma-}(u_h,\lambda_h)]_+\right>_{C}
 + \gamma^{-1} \|\bar e\|^2_{C} + s(\lambda_h, \gamma^{-1} \bar
 e ) 
\\ 
&\qquad  = 0
\end{align*}
where the last term vansihes for $k \ge 2$.
It follows that
\begin{align*}
& \gamma^{-1} \|\bar e\|^2_{C} - 
\frac14 |\lambda_h|_s^2
- C \gamma^{-1}\delta  h^2 \|\nabla (u -
 u_h)\|_{\Omega}^2
 \\
&\qquad \leq - \gamma^{-1}\left<\bar
   e,[P_{\gamma-}(u,\lambda)]_+-[P_{\gamma-}(u_h,\lambda_h)]_+\right>_{C}
\end{align*}
We recall that by the $L^2$-orthogonality there holds $\|\bar
e\|_{C}^2 = \|e\|_{C}^2 -   \|e -\bar e\|_{C}^2$
and therefore
\[
\gamma^{-\frac12} \|e\|^2_{C} \leq \gamma^{-\frac12} \|\bar
e\|^2_{C} + C \gamma^{-1} h^{2s} \|\nabla e\|^2_{\Omega}
\]
and consequently using also that $\|\bar e\|_{C} \leq
\|e\|_{C}$, there exists constants $C$, $c$ independent of $\gamma$
and $h$ such that
\begin{multline}\label{econt}
\frac12 \gamma^{-1} \| e\|^2_{C} - \frac12 \gamma^{-1}\|[P_{\gamma-}(u,\lambda)]_+-[P_{\gamma-}(u_h,\lambda_h)]_+\|^2_{C}\\
-\frac12 |\mu_h - \lambda_h|_s^2
- C( \delta + 1)\gamma^{-1} h^{2s} \|\nabla (u -
 u_h)\|_{\Omega}^2
\leq \frac12 |\mu_h|_s^2
\end{multline}
Collecting the results of equations \eqref{2nd_step},
and \eqref{econt} we have
\begin{align*}
&\left(\frac34 \alpha  - C (\delta +1)\gamma_0^{-1}\right)\|u-u_h\|_{H^1(\Omega)}^2 +
\frac14 |\mu_h - \lambda_h|^2_s 
\\
&\qquad \qquad  
+ \frac12 \gamma^{-1} \| e\|^2_{C}+ 
 \frac12 \gamma^{-1}\|[P_{\gamma-}(u,\lambda)]_+-[P_{\gamma-}(u_h,\lambda_h)]_+\|^2_{C}
 \\
&\qquad \qquad 
+\gamma^{-1}\left<(u-u_h),[P_{\gamma-}(u,\lambda)]_+-[P_{\gamma-}(u_h,\lambda_h)]_+\right>_{C}
\\
&\qquad 
\leq  \frac{1}{\alpha} \|u-v_h\|_{H^1(\Omega)}^2 +  \left<\lambda-\lambda_h,v_h-u\right>_{C}
\\
&\qquad \qquad 
-
\left<\mu_h-\lambda, u-u_h +[P_{\gamma-}(u,\lambda)]_+-[P_{\gamma-}(u_h,\lambda_h)]_+\right>_{C}+
2 |\mu_h|_s^2.
\end{align*}
Assuming that $C(\delta+1) \gamma^{-1} h^{2s} \leq \tfrac12 \alpha$, using that $\frac12 a^2
+ \frac12 b^2 + ab = \frac12 (a+b)^2$ and the Cauchy-Schwarz
inequality followed by the arithmetic-geometric inequality in the
second to last term in the right hand side we obtain
\begin{align}\nonumber
&\alpha \|u-u_h\|_{H^1(\Omega)}^2 
+
|\mu_h - \lambda_h|_s^2  
\\ \nonumber
&\qquad \qquad 
+ \gamma^{-1}\|(u-u_h)+[P_{\gamma-}(u,\lambda)]_+-[P_{\gamma-}(u_h,\lambda_h)]_+\|^2_{C}
\\ \label{3rd_step}
&\qquad \leq  \frac{4}{\alpha} \|u-v_h\|_{H^1(\Omega)}^2 +  4\left<\lambda-\lambda_h,v_h-u\right>_{C}
+
4\gamma\|\mu_h-\lambda\|^2_{C}+
8 |\mu_h|_s^2.
\end{align}
Observe that the $\delta$ in the condition may be omitted for $k \ge
2$.

It remains to control the Lagrange multiplier.
Observe that taking $v_h = R_h$ as defined in Lemma \ref{lem:lift} with $r_h = - \gamma 
I_{cf}\xi_h (\mu_h - \lambda_h)$ we may use Galerkin
orthogonality to obtain
\begin{align*}
&\gamma \|\xi^{\frac12}_h(\mu_h - \lambda_h)\|_{C}^2 - \gamma (\mu_h - \lambda_h, (1 -
I_{cf})\xi_h(\mu_h - \lambda_h))_{C}
\\
&\qquad \qquad +\gamma (\lambda - \mu_h, I_{cf}(\xi_h(\mu_h - \lambda_h)))_{C}
+a(u-u_h, R_h) 
\\
&\qquad = 0.
\end{align*}
Using the bound \eqref{eq:xicont}, $c^2_\xi \|\mu_h - \lambda_h\|_{C}^2
\leq \|\xi^{\frac12}_h(\mu_h - \lambda_h)\|_{C}^2$
we have
\begin{align*}
&\frac{c^2_\xi}{2} \gamma \|(\lambda - \lambda_h)\|_{C}^2 -
\frac{c_s^2}{c^2_\xi\delta} |\mu_h- \lambda_h|_s^2 -
\gamma\left(C +c_s^2 c^{-2}_\xi\right)\|\lambda - \mu_h\|_{C}^2
\\ 
&\qquad \qquad +a(u-u_h, R_h) 
\\
&\qquad \leq 0.
\end{align*}
Recall that by Lemma \ref{lem:lift} we have
\[
a(u-u_h, R_h)  \leq -\frac{ c^2_\xi}{ 4 } \gamma \|(\lambda -
\lambda_h)\|_{C}^2- \frac{c^2_\xi}{2} \gamma \|(\mu_h -
\lambda_h)\|_{C}^2 - \frac{C_R^2 \gamma}{c^2_\xi \alpha h^{2s}} \alpha\|\nabla (u - u_h)\|^2_\Omega
\]
from which we deduce that there exists a constant $C_{\lambda}>0$ such
that, assuming $\tfrac{\gamma}{\alpha h^{2s}} = \mathcal{O}(1)$.
\begin{equation}\label{lamcont}
\frac{c^2_\xi}{4 C_{\lambda}} \gamma \|(\lambda - \lambda_h)\|_{C}^2
- |\mu_h- \lambda_h|_s^2 - \alpha\|u - u_h\|^2_{H^1(\Omega)}\leq 
\gamma  \|\lambda - \mu_h\|_{C}^2.
\end{equation}
Multiplying both sides of \eqref{lamcont} by $1/2$ and adding it to
\eqref{3rd_step} leads to the inequality
\begin{align}\nonumber
&\frac12 \alpha \|u-u_h\|_{H^1(\Omega)}^2 +
\frac12 |\mu_h - \lambda_h|_s^2+c_\lambda \gamma \|(\lambda - \lambda_h)\|_{C}^2  
\\ \nonumber
&\qquad \qquad + \gamma^{-1}\|(u-u_h)+[P_{\gamma-}(u,\lambda)]_+-[P_{\gamma-}(u_h,\lambda_h)]_+\|^2_{C}
\\ \label{3rd_step2}
&\qquad 
\leq  \frac{C}{\alpha} \|u-v_h\|_{H^1(\Omega)}^2
+
C\gamma\|\mu_h-\lambda\|^2_{C}+
C |\mu_h|_s^2+  4\left<\lambda-\lambda_h,v_h-u\right>_{C}.
\end{align}
where $c_\lambda =\tfrac{c^2_\xi}{8
  C_{\lambda}}$. Finally splitting the last term on the right
hand side 
\[
4 \left<\lambda-\lambda_h,v_h-u\right>_{C} \leq \frac{c_\lambda}{2}
\gamma \|(\lambda - \lambda_h)\|_{C}^2 + 2 c_\lambda^{-1} \gamma^{-1} \|v_h-u\|^2_{C}
\]
we conclude that
\begin{align}\nonumber
&\alpha \|u-u_h\|_{H^1(\Omega)}^2 +
|\mu_h - \lambda_h|_s^2+c_\lambda \gamma \|(\lambda - \lambda_h)\|_{C}^2  
\\ \nonumber
&\qquad \qquad + \gamma^{-1}\|(u-u_h)+[P_{\gamma-}(u,\lambda)]_+-[P_{\gamma-}(u_h,\lambda_h)]_+\|^2_{C}
\\ \label{last_step}
&\qquad \lesssim  \frac{1}{\alpha} \|u-v_h\|_{H^1(\Omega)}^2
+\gamma\|\mu_h-\lambda\|^2_{C}+
|\mu_h|_s^2 +  4 \gamma^{-1} \|v_h-u\|^2_{C}.
\end{align}
\end{proof}
\begin{cor}
Assume that $u \in H^{r}(\Omega)$, $3/2< r \leq k+1$ and $\lambda \in
H^{r-1 - s}(\OmegaC)$, with $r-1 - s>0$ where $s=1/2$ for the Signorini problem and $s=1$
for the Obstacle problem and that $(u_h,\lambda_h)$ is the solution of
\eqref{FEM} with the contact operator defined by \eqref{infsup_stabform1} and under the same conditions on the parameters as in Theorem
\ref{apriori_error_2}. Then there holds
\begin{align*}\nonumber
&\alpha \|u-u_h\|_{H^1(\Omega)}  
+ \gamma\| (\lambda -
  \lambda_h)\|_{C} 
\\ \nonumber 
&\qquad \qquad  + \gamma^{-1/2}\|(u-u_h)+[P_{\gamma-}(u,\lambda)]_+-[P_{\gamma-}(u_h,\lambda_h)]_+\|_{C} 
\\
&\qquad \lesssim h^{r-1} (|u|_{H^{r}(\Omega)} + |\lambda|_{H^{r-1 - s}(\OmegaC)}).
\end{align*}
\end{cor}
\begin{proof}
Similar to that of Corollary \ref{cor:error_est}.
\end{proof}
\section{Numerical examples}

In the numerical examples below, we define $h=1/\sqrt{\text{NNO}}$, where NNO denotes the number of nodes in a uniformly refined mesh.
We use the formulation \eqref{CompFEM} with the nonlinear term defined
by \eqref{stab_form1}. For the spaces we chose piecewise linear finite elements for the primal variable and piecewise constants for the Lagrange multipliers, constant per element for the obstacle problem, and
constant per element edge on the Signorini boundary for the Signorini problem.

\subsection{Smooth obstacle problem}

Our smooth obstacle example, adapted from \cite{NoSiVe03}, is posed on the square $\Omega=(-1,1)\times(-1,1)$ with $\psi=0$ and
\[
f=\left\{\begin{array}{c}
    8 r_0^2 (1-(r^2-r_0^2)) \quad \text{if $r\leq r_0$},\\
  8 (r^2+ (r^2-r_0^2)) \quad \text{if $r> r_0$},
\end{array}\right.
\]
where $r=\sqrt{x^2+y^2}$ and $r_0=1/4$, and
with Dirichlet boundary conditions taken from the corresponding exact solution 
\[
u = -[r^2-r_0^2]_+^2 .
\]

We choose $\gamma=\gamma_0 h$ with $\gamma_0 = 1/100$ and show the convergence in the $L_2$-- and $H^1$--norms in Figure \ref{fig:smooth}. An elevation of the computed solution on one of the meshes in a sequence is given in Fig. \ref{fig:elevationsmooth}. We note the optimal convergence of $O(h^2)$ in $L_2$ (dashed line has inclination 2:1) and $O(h)$ in $H^1$
(dotted line has inclination 1:1). 

\subsection{Nonsmooth obstacle problem}

This example was proposed by Braess et al. \cite{BrCaHo07}.
The domain is $\Omega= (-2, 2)\times (-2, 2) \setminus [0, 2)\times (-2, 0]$ with $\psi=0$ and
\[
f(r,\varphi) = r^{2/3}\sin{(2\varphi/3)}(\gamma'(r)/r-\gamma''(r))+\frac{4}{3}r^{-1/3}\gamma'(r)\sin(2\varphi/3)+\gamma_2(r)
\]
where, with $\hat{r}=2(r-1/4)$,
\[
\gamma_1(r)=\left\{\begin{array}{ll}1,& \hat{r} < 0\\
-6\hat{r}^5+15\hat{r}^4-10\hat{r}^3+1, & 0\leq\hat{r}<1\\
0, & \hat{r}\geq 1,\end{array}\right.
\]
\[
\gamma_2(r)=\left\{\begin{array}{ll}0,& r\leq 5/4 ,\\
1 & \text{elsewhere.}\end{array}\right.
\]
with Dirichlet boundary conditions taken from the corresponding exact solution 
\[
u(r,\varphi)=-r^{2/3}\gamma_1(r)\sin(2\varphi/3)
\]
which belongs to $H^{5/3-\varepsilon}(\Omega)$ for arbitrary $\varepsilon > 0$.

For this example we plot, in Fig. \ref{fig:nonsmootherror}, the error on consecutive refined meshes.
We note the suboptimal convergence in $L_2$ which agrees with the regularity of the exact solution. In Fig. \ref{fig:elnonsmo} we show an elevation of the approximate solution on one of the meshes used to compute convergence.

\subsection{Signorini problem}

The Signorini problem is posed on the unit square $(0,1)\times (0,1)$ with homogeneous Dirichlet boundary conditions at $y=1$, 
homogeneous Neumann boundary conditions at $x=0$ and $x=1$, and a Signorini boundary at $y=0$. The load is $f=-2\pi\sin{2\pi x}$ (following \cite{BB00}), and we set $\gamma_0=0.1$. No explicit solution is available and we instead use an overkill solution, using 66049 nodes (corresponding to $h\approx 4\times10^{-3})$ to estimate the error. In Fig. \ref{fig:signorinierror} we show the convergence in the $L_2$-- and $H^1$--norms and again
we observe optimal convergence of $O(h^2)$ in $L_2$ (dashed line has inclination 2:1) and $O(h)$ in $H^1$ (dotted line has inclination 1:1). Finally, in Fig. \ref{fig:elsigno} we show an elevation of the computed solution.

\section*{Acknowledgments}

This research was supported in part by EPSRC, UK \newline (EP/J002313/1), the 
Swedish Foundation for Strategic Research (AM13-0029), the Swedish Research 
Council (2011-4992, 2013-4708), and the Swedish Strategic Research Program 
Essence.

The first author  wishes to thank Dr. Franz Chouly,
Prof. Patrick Hild and Prof. Yves Renard for interesting discussions on Nitsche's
method for contact problems and the augmented Lagrangian method. 

%\section*{References}
\bibliographystyle{plain}
\bibliography{contact}
%\end{document}

%
\newpage
\begin{figure}[h]
\begin{center}
\includegraphics[height=7cm]{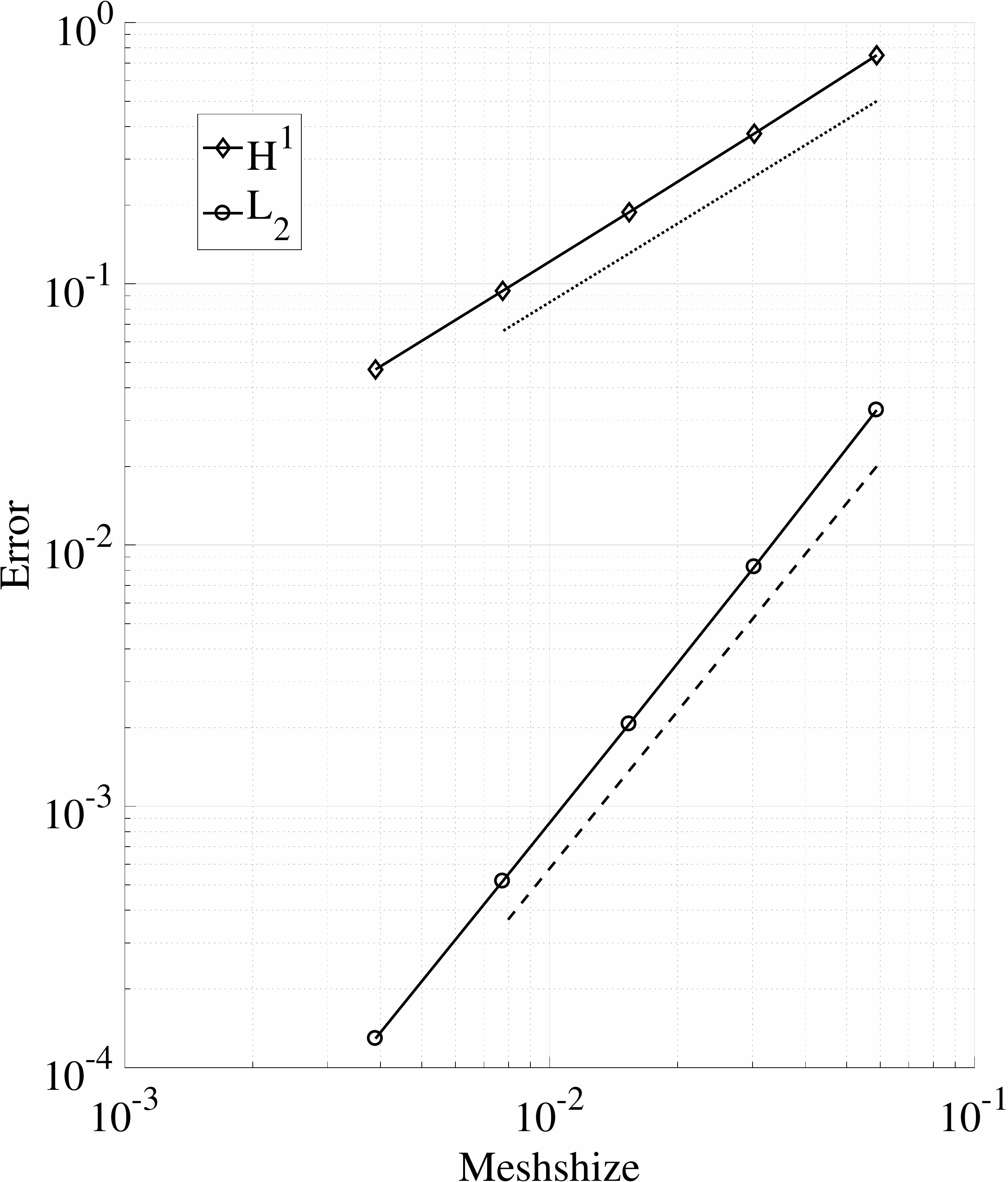}
\end{center}
\caption{Convergence for the smooth obstacle. Dotted line has inclination 1:1, dashed line has inclination 1:2.\label{fig:smooth}}
\end{figure}
\begin{figure}[h]
\begin{center}
\includegraphics[height=6cm]{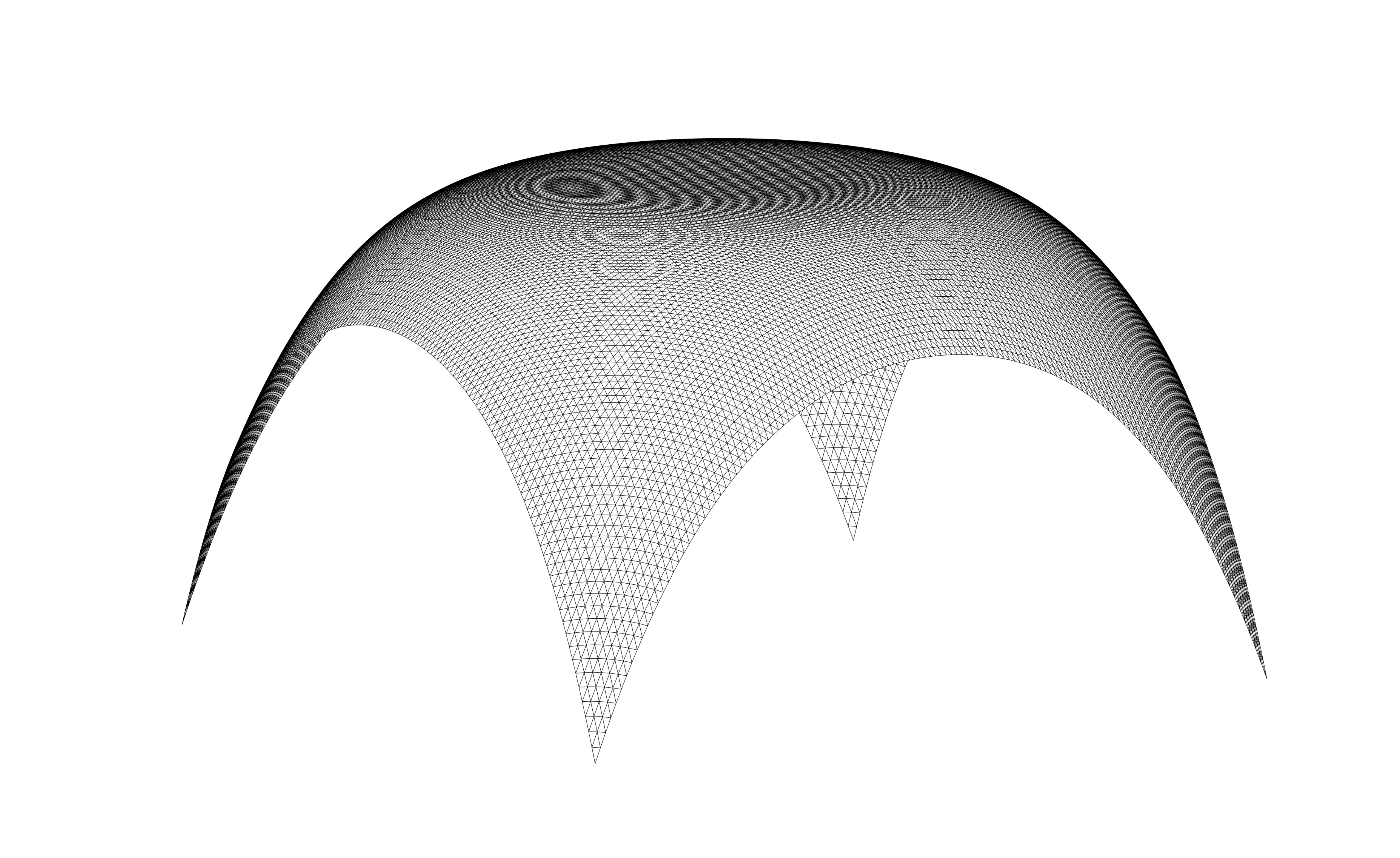}
\end{center}
\caption{Elevation of the discrete solution, smooth obstacle.\label{fig:elevationsmooth}}
\end{figure}

\begin{figure}[h]
\begin{center}
\includegraphics[height=7cm]{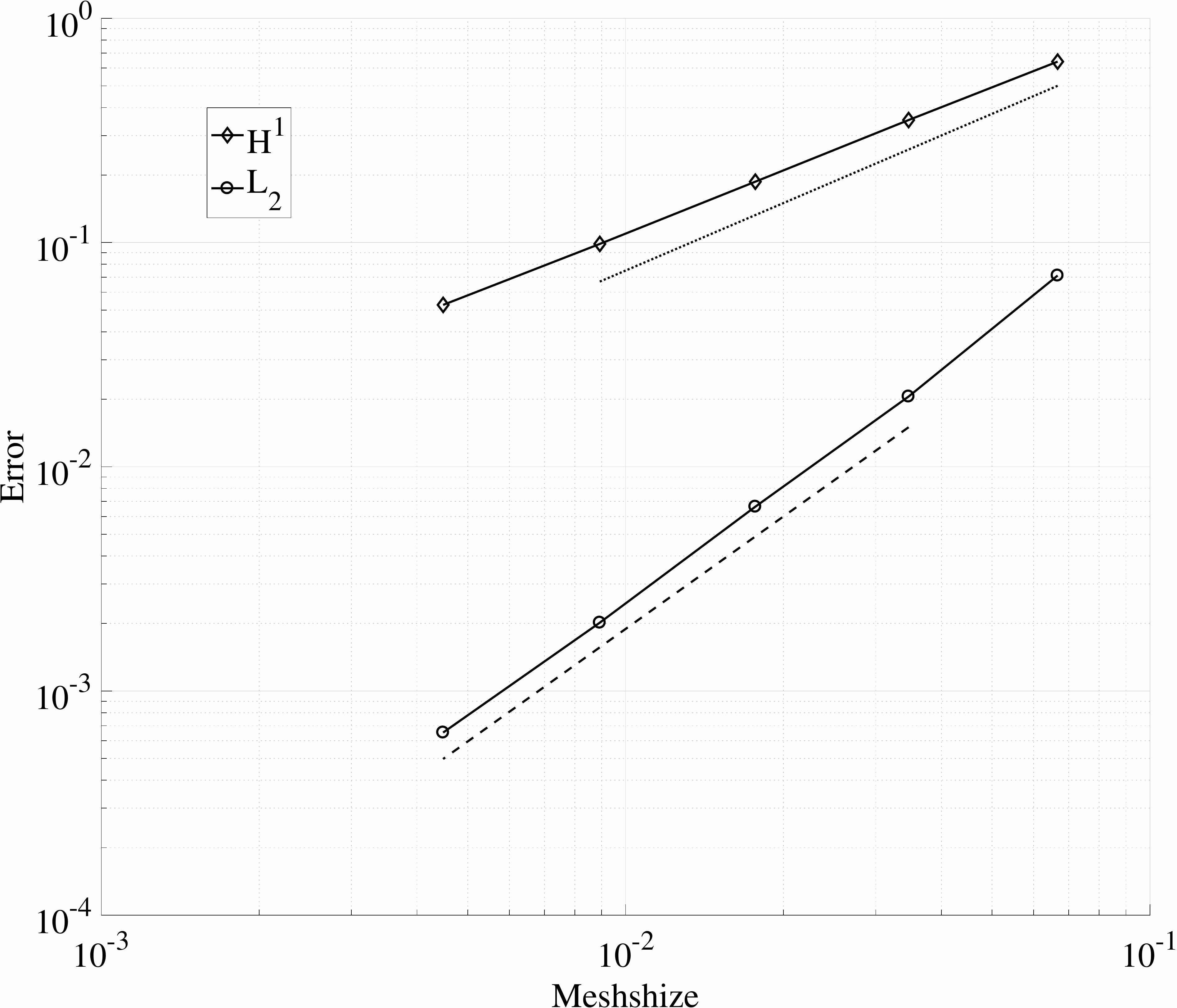}
\end{center}
\caption{Convergence for the nonsmooth obstacle. Dotted line has inclination 1:1, dashed line has inclination 1:5/3.\label{fig:nonsmootherror}}
\end{figure}
\begin{figure}[h]
\begin{center}
\includegraphics[height=6cm]{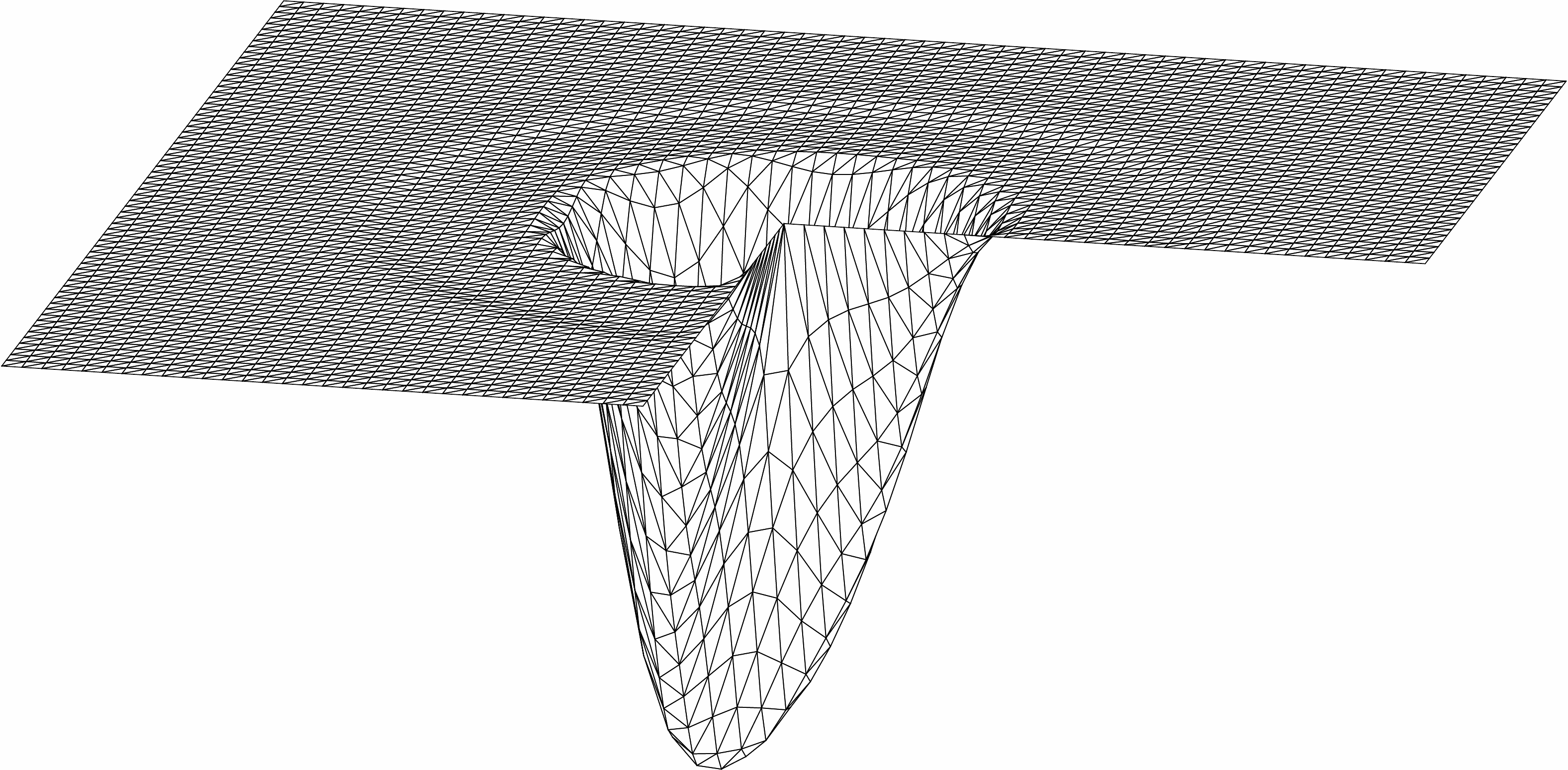}
\end{center}
\caption{Elevation of the discrete solution, nonsmooth obstacle.\label{fig:elnonsmo}}
\end{figure}

\begin{figure}[h]
\begin{center}
\includegraphics[height=12cm]{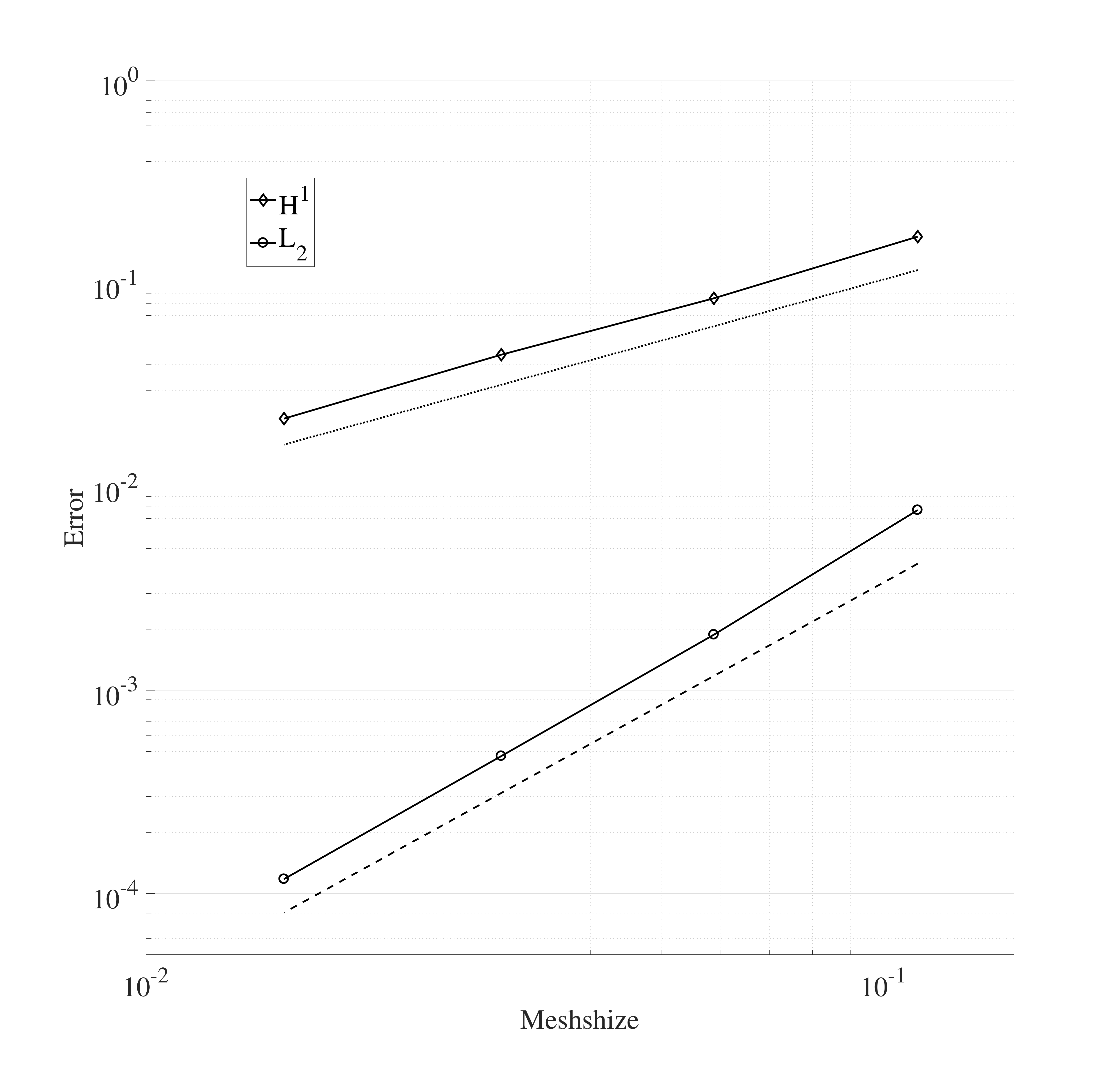}
\end{center}
\caption{Convergence for the Signorini case.\label{fig:signorinierror}}
\end{figure}
\begin{figure}[h]
\begin{center}
\includegraphics[height=8cm]{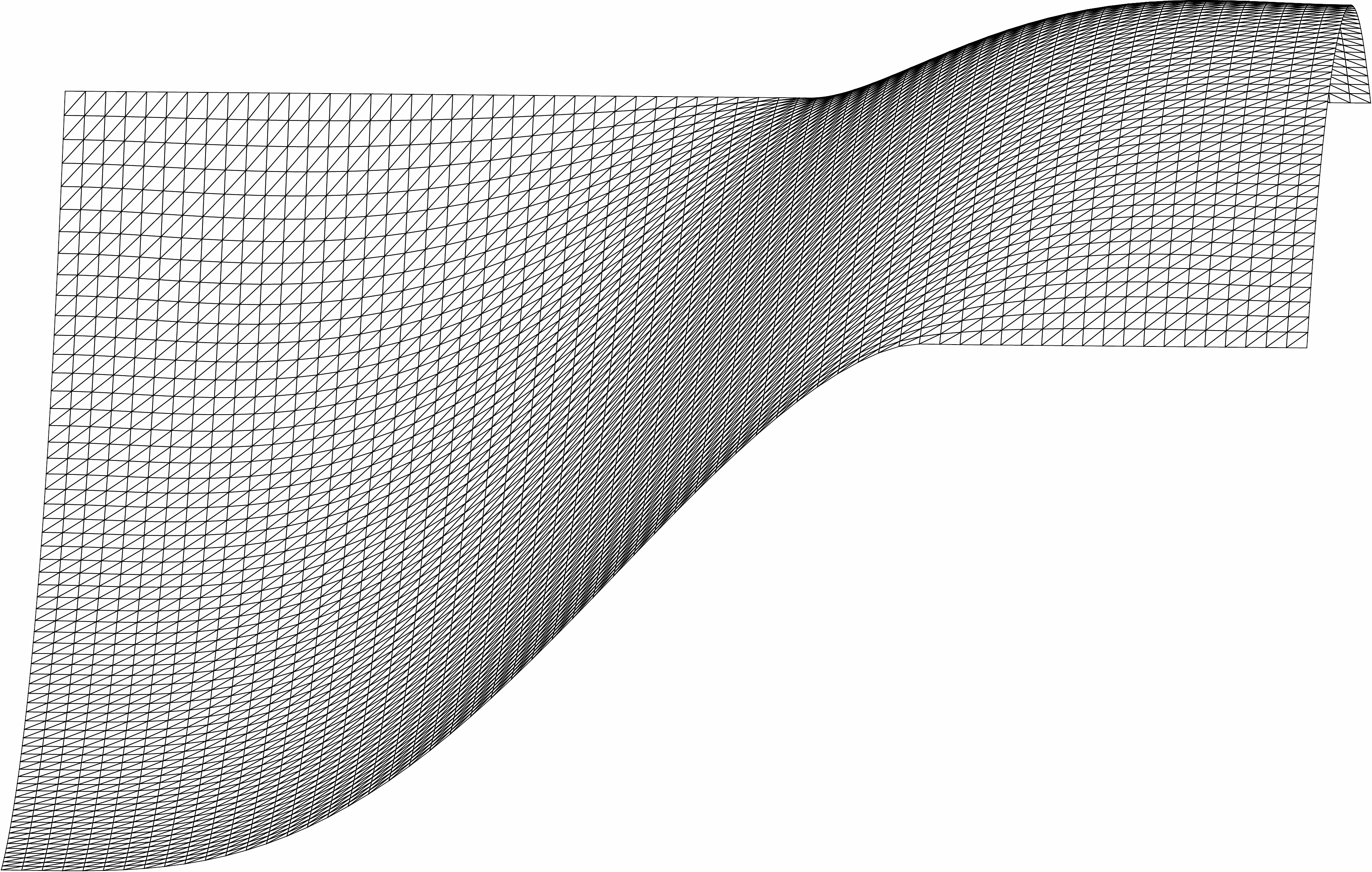}
\end{center}
\caption{Elevation of the discrete solution, Signorini case.\label{fig:elsigno}}
\end{figure}

\end{document}